\documentclass{amsart}

\usepackage[utf8]{inputenc}
\usepackage[english]{babel}
\usepackage[margin=1.2in]{geometry}
\usepackage{amsmath, amsthm, amscd, amssymb, mathtools}
\usepackage{verbatim} 
\usepackage{wrapfig}
\usepackage{todonotes}
\usepackage{bbm}

\usepackage[nocompress]{cite}

\usepackage{hyperref}
\hypersetup{
  colorlinks   = true, 
  urlcolor     = blue, 
  linkcolor    = blue, 
  citecolor   = red 
}

\usepackage{tikz-cd}
\tikzset{
  symbol/.style={
    draw=none,
    every to/.append style={
      edge node={node [sloped, allow upside down, auto=false]{$#1$}}}
  }
}
\tikzset{
    labl/.style={anchor=south, rotate=90, inner sep=.5mm}
}

\usepackage{enumitem}
\setenumerate{itemsep=0pt,topsep=3pt}
\setenumerate[1]{label=\textup{(\roman*)}}

\theoremstyle{plain}
\newtheorem{theorem}{Theorem}[section]
\newtheorem{lemma}[theorem]{Lemma}
\newtheorem{proposition}[theorem]{Proposition}
\newtheorem{corollary}[theorem]{Corollary}

\theoremstyle{definition}
\newtheorem{definition}[theorem]{Definition}
\newtheorem{remark}[theorem]{Remark}
\newtheorem{example}[theorem]{Example}

\DeclareMathOperator{\Hom}{Hom}

\DeclareMathOperator{\Aut}{Aut}
\DeclareMathOperator{\End}{End}
\newcommand{\id}{\mathrm{id}}
\DeclareMathOperator{\im}{im}

\DeclareMathOperator{\Id}{Id}
\DeclareMathOperator{\Dec}{Dec}
\let\Im=\relax
\DeclareMathOperator{\Im}{lim\,im}
\DeclareMathOperator{\Ker}{lim\,ker}
\DeclareMathOperator{\Tr}{Tr}

\newcommand{\mapsfrom}{\mathrel{\reflectbox{\ensuremath{\mapsto}}}}
\newcommand{\tensor}{\otimes}

\newcommand{\D}{\ensuremath{\mathcal{D}}}
\DeclareMathOperator{\GpRg}{\D}

\newcommand{\Z}{\ensuremath{\mathbb{Z}}}
\newcommand{\Q}{\ensuremath{\mathbb{Q}}}

\newcommand{\C}{\ensuremath{\mathbb{C}}}
\newcommand{\F}{\ensuremath{\mathbb{F}}}
\newcommand{\isom}{\xrightarrow{\sim}}

\newcommand{\p}{\ensuremath{\mathfrak{p}}}
\renewcommand{\i}{\textup{i}}

\newcommand{\catt}[1]{\textbf{\textup{#1}}}
\newcommand{\homcat}[1]{\ensuremath{#1}\textup{Hom}}
\newcommand{\ltm}{\ensuremath{\big(\begin{smallmatrix}\Z&0\\ \Z&\Z\end{smallmatrix}\big)}}
\newcommand{\allm}{\ensuremath{\big(\begin{smallmatrix}\Z&\Z\\ \Z&\Z\end{smallmatrix}\big)}}

\title
[Realizing orders as group rings]
{Realizing orders as group rings}
\author[H.\ W.\ Lenstra, Jr.]{H.\ W.\ Lenstra, Jr.}
\address{Mathematisch Instituut, Universiteit Leiden, The Netherlands}
\email{hwl@math.leidenuniv.nl}
\author[A.\ Silverberg]{A.\ Silverberg}
\address{Department of Mathematics, University of California, Irvine, CA 92697, USA}
\email{asilverb@uci.edu}
\author[D.\ M.\ H. van Gent]{D.\ M.\ H. van Gent}
\address{Mathematisch Instituut, Universiteit Leiden, The Netherlands}
\email{d.m.h.van.gent@math.leidenuniv.nl}

\subjclass[2010]{13A02}
\keywords{group rings}
\thanks{Support for the research was provided by the Alfred P.~Sloan Foundation
and the National Science Foundation. Lenstra thanks the Max Planck Institute for Mathematics in Bonn for the hospitality and support extended to him during the fall of 2016.
Silverberg thanks the Simons Laufer Mathematical Sciences Institute for support while in residence during the Spring 2023 semester. We thank the referee for helpful comments.}

\begin{document}

\begin{abstract} 
An order is a commutative ring that as an abelian
group is finitely generated and free. A commutative ring is reduced if it has no non-zero nilpotent elements.
In this paper we use a new tool, namely, the fact that every reduced order has a universal grading, to answer questions about realizing orders as group rings. In particular, we address the Isomorphism Problem for group rings in the case where the ring is a reduced order.
We prove that any non-zero reduced order $R$ can be written as a group ring in a unique ``maximal'' way, up to
isomorphism. More precisely, there exist a ring $A$ and a finite abelian group $G$, both uniquely determined up to
isomorphism, such that $R\cong A[G]$ as rings, and such that
if $B$ is a ring and $H$ is a group, then $R\cong B[H]$ as
rings if and only if there is a finite abelian group $J$ such
that $B\cong A[J]$ as rings and $J\times H\cong G$ as groups. 
Computing $A$ and $G$ for given \(R\) can be done by means of an algorithm that is not quite polynomial-time.
We also give a description of the automorphism group of $R$ in terms of
$A$ and~$G$.
\end{abstract}

\maketitle


\section{Introduction}

Consider the map  
\[ (\{ \textup{rings} \}/{\cong}) \times (\{\textup{groups}\}/{\cong}) \to (\{\textup{rings}\}/{\cong}) \]
that for each ring $A$ and each group $G$ 
sends the pair of isomorphism classes of $(A, G)$ to the isomorphism class of the group ring $A[G]$. 
The {\em Isomorphism Problem for group rings} asks about the fibers of this map,
i.e., given a ring $R$, what can one say about the pairs $(A, G)$ with $A[G] \cong R$? 
There is a great deal of literature on this subject,
starting with the 1940 paper by Higman \cite{Higman}, 
which solved the case where \(A=\Z\) and \(G\) is a finite abelian group, 
and including \cite[Chapter~14]{Passman}, \cite[Chapters III and~IV]{Sehgal}, \cite{Sehgal2}, and the recent survey article \cite{Survey2022} and the references therein.

The emphasis has been 
on results stating that the fibers of our map are often ``small'' in a suitable sense.
For instance, it is a consequence of a theorem of May \cite{May2} that if $G$ is an abelian group and $\Z[G] \cong \Z[H]$ then $G \cong H$.
This contrasts with the complex case, where $\C[G] \cong \C[H]$, even as \(\C\)-algebras, whenever $G$ and $H$ are finite abelian groups of the same order.

One can slightly refine the question 
by not just asking for the existence of an isomorphism from $A[G]$ to $R$, but asking for the isomorphism as well. 
For a given non-zero ring $R$ the object of study is thus the set of triples $(A,G,\phi)$, where $A$ is a ring (always with \(1\)), $G$ is a group, and $\phi\colon A[G] \to R$ is a ring isomorphism. 
This essentially comes down to considering the set
\[\GpRg(R) = \{(A,G) : \textup{$A\subset R$ a subring, $G\subset R^*$ a subgroup, with $A[G] \to R$ a ring isomorphism}\},\]
where $R^*$ is the group of units of $R$, and $A[G] \to R$ is the natural group homomorphism.

One of the main contributions of the present paper is a new group-theoretic description of the set $\GpRg(R)$ for an important class of rings $R$, namely for connected reduced orders; see below for definitions. 
These rings are commutative, so $A$ is commutative and $G$ is abelian. 
The case of commutative group rings has received special attention in the literature (for example \cite{May1,May2,May3,May4,Adjaero,Parmenter}), an important tool being the abelian group $\mu = \mu(R)=\{\zeta\in R : (\exists\, n \in\Z_{\geq 1})\ \zeta^n=1\}$ of {\em roots of unity\/} in a commutative ring \(R\).

For connected reduced orders $R$ we add a new tool, namely the universal grading of $R$ and the abelian group $\Gamma = \Gamma(R)$ by which it is graded. 
The existence of such a universal grading was recently established in \cite{UniversalGradings}, and it was proved there that it comes with a natural map $d = d_R\colon \mu \to \Gamma$.
In one of our main theorems (Theorem~\ref{thm:S0_to_S} below) we exhibit, for connected reduced orders $R$, a natural bijection 
\[\GpRg(R) \isom \{f \in \Hom(\Gamma,\mu) : f \circ d \circ f = f\}.\]
Thus, the set of realizations of a connected reduced order as a group ring is parametrized by an easily defined set of homomorphisms from $\Gamma$ to $\mu$, both of which are finite abelian groups.

The theorem just stated has several striking consequences in the context of the Isomorphism Problem. For example, in  Theorem~\ref{thm:gcd} we prove that among reduced orders, group rings can only be isomorphic if they are so for obvious reasons. 
In Theorem~\ref{thm:universal} we show that each non-zero reduced order can in a unique ``maximal'' way be realized as a group ring.
In Theorem~\ref{thm:swop_grp} we establish a curious ``cross-over'' result.
For the connected case, we describe in Theorem~\ref{thm:matrix_group} the group of ring automorphisms of such a maximal group ring.

Deducing these results from our description of $\GpRg(R)$ is surprisingly non-trivial, and has two features that may be unexpected
in the context of
commutative algebra. The first is the use of modules over non-commutative rings, and the second the use of
techniques from number theory, taken from \cite{RootsOfUnity} and \cite{UniversalGradings}. 
We next give precise statements of some of our main results, along with a brief introduction to our new tool and methods.

By an {\em order} we mean a commutative ring of which the additive group is isomorphic to $\Z^n$ for some $n\in\Z_{\geq 0}$.
A ring element $x$ is {\em nilpotent} if $x^n = 0$ for some $n\in\Z_{>0}$. 
We call a commutative ring {\em reduced} if it has no non-zero nilpotent elements.

Obviously, if $A$ is a ring and $I$ and $H$ are
groups, then the group rings $A[I\times H]$ and $(A[I])[H]$
are isomorphic as rings. The following result expresses that,
among reduced orders, group rings can only be isomorphic if
they are so for this obvious reason.

\begin{theorem}\label{thm:gcd}
Suppose \(A\) and \(B\) are reduced orders and \(G\) and \(H\) are finite abelian groups. Then the following are equivalent:
\begin{enumerate}
\item \(A[G]\cong B[H]\) as rings,
\item there exist an order \(C\) and finite abelian groups \(I\) and \(J\) such that \(A\cong C[I]\) and \(B\cong C[J]\) as rings and \(I\times G\cong J\times H\) as groups.
\end{enumerate}
\end{theorem}

The proof is given in Section \ref{sec:general_case}.

We call a commutative ring $R$ {\em stark} if there do not exist a ring \(A\) and a non-trivial group \(G\) such that \(R\) is isomorphic to the group ring \(A[G]\).

\begin{theorem}\label{thm:universal}
Let \(R\) be a non-zero reduced order. Then there exist a stark ring \(A\), unique up to ring isomorphism, and a finite abelian group \(G\), unique up to group isomorphism, such that \(R\cong A[G]\) as rings.
\end{theorem}

Theorem~\ref{thm:universal}
is an immediate consequence of Theorem~\ref{thm:gcd}, and we give a proof in Section~\ref{sec:general_case}.
We note that Theorem~\ref{thm:gcd} may be deduced from Theorem~\ref{thm:universal} using that group rings of finite
abelian groups over reduced orders are reduced (see Proposition~\ref{prop:reduced_connected_groupring}(i)).
There are many examples showing that $A$ and $G$ in Theorem \ref{thm:universal} need not be uniquely
determined as a subring of $R$ and a subgroup of 
$R^*$,  
respectively (see Example~\ref{ex:list_all_of_gprg}).

Theorems~\ref{thm:gcd} and \ref{thm:universal} are closely related to known results about group rings of torsion abelian groups over integral domains of characteristic zero; see Corollary~5 in \cite{Adjaero} and Theorem~8 in \cite{Parmenter}, which were proved by studying the group of torsion units in the group ring.
Our new tool, namely the consideration of gradings, allows us in the case of orders to replace the condition that the base ring be a domain by the weaker condition that it be stark and reduced.

A ring element \(x\) is {\em idempotent} if \(x^2=x\). 
We call a commutative ring {\em connected} if it has precisely two idempotents, i.e., the ring is non-zero and \(0\) and \(1\) are the only idempotents.
For a commutative ring \(R\) with a 
subring \(A\subset R\) and a subgroup \(G\subset R^*\) we write \(A[G]=R\) when the natural map \(A[G]\to R\) is a ring isomorphism. 

If \(R\), \(A\), and \(G\) are as in Theorem~\ref{thm:universal}, then \(A\) is isomorphic to a subring of \(R\), and \(G\) is isomorphic to a subgroup of \(\mu(R)\),
but as we remarked, this subring and subgroup need not be uniquely determined.
However, in the important special case where $R$ is connected, there is a sense in which the set of subrings \(A\subset R\) that can be used and the set of subgroups \(G\subset\mu(R)\) that can be used are entirely independent.
The following ``cross-over'' result, which we found surprising, formulates this more precisely.

\begin{theorem}\label{thm:swop_grp}
Let \(R\) be a connected reduced order. Suppose that \(A\) and \(B\) are stark subrings of \(R\)
and that \(G\) and \(H\) are subgroups of \(R^*\) such that  
\(A[G]=B[H]=R\). Then  
\(A[H]=B[G]=R\).
\end{theorem}

Theorem~\ref{thm:swop_grp} 
will follow immediately from Theorem~\ref{thm:strong_swop_grp}, which is proved in Section~\ref{sec:degree_map} below. 
As can be seen in Example~\ref{ex:counter_connected_swop}, we cannot drop the assumption
that \(R\) be connected in Theorem~\ref{thm:swop_grp}.

We call a ring element \(x\) {\em autopotent\/} if \(x^{n+1}=x\)
for some \(n\in\Z_{>0}\), or equivalently if it is the product of a root of unity and an idempotent that
commute with each other (see Proposition~\ref{prop:autopotents}.iv).

We have the following algorithmic result. All our algorithms will be deterministic.

\begin{theorem}\label{thm:algos}
There is 
an algorithm that, given a non-zero reduced order \(R\), computes a stark subring \(A\subset R\) and a subgroup \(G\subset\mu(R)\) such that \(A[G]=R\).
This algorithm runs \textup{(a)} 
in polynomial time when the additive group of \(R\) is generated by autopotents,
and generally \textup{(b)} in time \(n^{O(m)}\) where \(n\) is the length of the input and \(m\) is the number of minimal prime ideals of~\(R\).
\end{theorem} 

The proof of Theorem~\ref{thm:algos}, and a description of the algorithm and 
how its input and output are given, 
are found in Section~\ref{sec:algo}.
Note that the algorithm runs in polynomial time when \(m\) is bounded by a constant. 
The case \(m=1\) is precisely the case where \(R\) is a domain, in which case one necessarily has \(A=R\) and \(G=1\).  
A notable special case for (a) is when \(R\) is the product of finitely many group rings over~\(\Z\).
We do not know whether there exists a polynomial-time algorithm that decides whether a given
reduced order is stark.

We next discuss the new tool that we announced.
A {\it grading\/} of a commutative ring $R$ is a pair $(\Delta,
(R_\delta)_{\delta\in\Delta})$ where $\Delta$ is an abelian
group and the $R_\delta$ are additive subgroups of $R$ such
that, first, for all $\gamma$, $\delta\in\Delta$ one has
$R_\gamma\cdot R_\delta\subset R_{\gamma\delta}$, and, second,
one has $\bigoplus_{\delta\in\Delta} R_\delta=R$ in the sense that the
natural map $\bigoplus_{\delta\in\Delta} R_\delta\to R$ is bijective.
For each commutative ring \(A\) and abelian group \(G\) the ring \(A[G]\) has a natural grading \((G, (A\gamma)_{\gamma\in G})\).

Suppose now that $R$ is a reduced order. 
Two recent results from \cite{UniversalGradings}  
(Theorems~\ref{thm:universal_main_thm1} and \ref{thm:universal_main_thm2} below) are of crucial importance to the present paper. 
The first states that $R$ has a grading
$(\Gamma,(R_\gamma)_{\gamma\in\Gamma})$ that is universal in
the sense that giving a grading of $R$ with group $\Delta$ is
equivalent to giving a group homomorphism $\Gamma\to\Delta$
(see Definition~\ref{def:univgrading}). The abelian group $\Gamma=\Gamma(R)$ is
finite, and by universality it is canonically defined, so that
the group $\Aut(R)$ of ring automorphisms acts on it (see Remark~\ref{rem:r2d2}). 
The second result states that, if $R$ is
also connected, then for each $\zeta\in\mu$
there is a unique $\gamma\in\Gamma$ such that $\zeta\in
R_\gamma$, and we write $d(\zeta)=\gamma$; this gives rise to a
group homomorphism $d\colon\mu\to\Gamma$, sometimes denoted \(d_R\), which we call the
{\it degree map}.

For each $(A,G)\in\GpRg(R)$, the natural grading of \(A[G]\) gives rise to the grading \((\mu,(A_\zeta)_{\zeta\in\mu})\) of \(R\), where $A_\zeta=A\zeta$ if $\zeta\in G$ and $A_\zeta=0$ otherwise.
This grading corresponds to a homomorphism \(f\colon\Gamma\to\mu\).
We prove in Theorem~\ref{thm:S0_to_S} that in the connected case this construction induces a bijection from 
$\GpRg(R)$ to the set of 
$f$ 
such that \(fdf=f\), so $\GpRg(R)$ can be studied group-theoretically as mentioned above.

Our next result is on automorphism groups of orders. As we saw in Theorem~\ref{thm:universal},
any non-zero reduced order $R$ may be written as $A[G]$, where $A$ is a stark subring of $R$ and $G$ is a
subgroup of $\mu(R)$, but it is only up to the action of the automorphism group $\Aut(R)$ that the pair $(A,G)$ is unique.
Thus, a natural question is how to describe the group $\Aut(R)$ in terms of $A$ and~$G$.
Restricting to connected orders, which is the main case of interest, we show in Theorem~\ref{thm:matrix_group}
that this group has a somewhat curious description in terms of $2\times2$ matrices.

We write \(\mu=\mu(A)\) and \(\Gamma=\Gamma(A)\).
There is a left action of \(\Aut(A)\) on \(\Hom(G,\mu)\) given via the restriction \(\Aut(A)\to\Aut(\mu)\), and a right action of \(\Aut(A)\) on \(\Hom(\Gamma,G)\) induced by its action on \(\Gamma\).
For abelian groups \(M\) and \(N\) we will, here and elsewhere, write the group \(\Hom(M,N)\) additively, regardless of the notation used for~\(N\).

\begin{theorem}\label{thm:matrix_group}
Let \(A\) be a stark connected reduced order with degree map \(d_A\colon\mu\to\Gamma\) and let \(G\) be a finite abelian group. 
We equip the cartesian product
\begin{align*} 
M = \begin{pmatrix} \Aut(A) & \Hom(G,\mu) \\ \Hom(\Gamma,G) & \Aut(G) \end{pmatrix}
\end{align*}
of \(\Aut(A)\), \(\Hom(G,\mu)\), \(\Hom(\Gamma,G)\), and \(\Aut(G)\) with the following multiplication:
\begin{align*}
\begin{pmatrix} \alpha_1 & s_1 \\ t_1 & \sigma_1 \end{pmatrix} \begin{pmatrix} \alpha_2 & s_2 \\ t_2 & \sigma_2 \end{pmatrix} &= \begin{pmatrix} \alpha_1 \alpha_2 + s_1 t_2 & \alpha_1 s_2 + s_1 \sigma_2 \\ t_1\alpha_2+\sigma_1 t_2 & t_1d_As_2+\sigma_1\sigma_2  \end{pmatrix},
\end{align*}
where the sum in \(\Aut(A)\) is as in Lemma~\ref{lem:Hom_action_on_Aut} and the sum in \(\Aut(G)\) is taken inside \(\End(G)\).
For \(x\in A\) and \(g\in G\) write \((\begin{smallmatrix}x \\ g \end{smallmatrix})\) for the element \(x\cdot g \in A[G]\). 
Then:
\begin{enumerate}
\item
\(M\) is a group; 
\item
there is a group isomorphism  \(M\isom\Aut(A[G])\) such that the action of \(M\) on the ring \(A[G]\) is given by
\begin{align*}
 \begin{pmatrix} \alpha & s \\ t & \sigma \end{pmatrix} \begin{pmatrix} x \\ g \end{pmatrix} &= \begin{pmatrix} \alpha(x) \cdot s(g) \\  t(\gamma) \cdot \sigma(g) \end{pmatrix}
\end{align*}
for all $g\in G$, $\gamma\in\Gamma$, and $x\in A_\gamma$.
\end{enumerate}
\end{theorem}

For the proof and additional results, we refer to Section~\ref{sec:automorphisms}.

\begin{theorem}\label{thm:matrix_group_corollary}
Let \(A\) be a stark connected reduced order with degree map \(d_A\colon\mu\to\Gamma\) and let \(G\) be a finite abelian group.
Let 
$\mathcal{A} = \{ \bigoplus_{\gamma\in\Gamma} A_\gamma t(\gamma)  : t \in \Hom(\Gamma,G) \},$ 
a set of subrings of \(A[G]\).
Let
$\mathcal{G} = \{ \{s(g)g : g\in G\} : s\in \Hom(G,\mu) \},$
a set of subgroups of $A[G]^*$. 
Then
$$\mathcal{A}\times\mathcal{G} = \{ (B,H) \in\mathcal{D}(A[G]) : \text{\(B\) is stark}\}.$$
\end{theorem}

The pairs $(B,H)$ in Theorem~\ref{thm:matrix_group_corollary}  are exactly the pairs occurring in Theorem~\ref{thm:swop_grp} with \(R=A[G]\).
Theorem~\ref{thm:matrix_group_corollary} allows one, in the connected case, to read off the set of subrings \(A\subset R\) and the set of subgroups \(G\subset R^*\) that can be used in Theorem~\ref{thm:universal}. At the end of Section~\ref{sec:general_case} we show how Theorem~\ref{thm:matrix_group_corollary} readily follows from Theorems~\ref{thm:matrix_group} and \ref{thm:universal}.

Our results suggest several questions. 
Which other rings have a universal grading?
Under what conditions is there a degree map? 
Once one has a degree map $d\colon \mu \to \Gamma$, are there consequences for the Isomorphism Problem for group rings, even if $\mu$ or $\Gamma$ is infinite? 
Can we replace rings by algebras over base rings other than \(\Z\)?
For preliminary results in these directions, see Theorem~6.21 in \cite{Daan} (cf.\ Theorem~\ref{thm:universal_main_thm1}), and \cite{Daan2} (cf.\ Theorem~\ref{thm:universal_main_thm2}).

The structure of the paper is as follows. 
Section~\ref{sec:Modules} contains generalities on modules of finite length over rings that need not be commutative.
We show in Section~\ref{sec:decomp_new} that a morphism of abelian groups, in particular the degree map defined above, can be interpreted as a module over a certain matrix ring.
This enables us, in Section~\ref{sec:Ustar}, to apply the theorem of Krull--Remak--Schmidt to morphisms of finite abelian groups.
Additionally, we introduce, for any morphism of abelian groups, an important group \(U^*\) that acts on it.  
In Section~\ref{sec:GroupRings} we treat generalities on graded rings and group rings.
In Section~\ref{sec:degree_map} we apply the theory from the former sections to the degree map of a connected reduced order.
It is an essential property of \(U^*\) that its action on the degree map can be lifted to an action on the order (see Lemma~\ref{lem:group_actions}).
We prove Theorem~\ref{thm:swop_grp} by showing that the pairs \((A,G)\), \((B,H)\), \((A,H)\), and \((B,G)\) are in the same orbit under the action of \(U^*\). This effectively proves Theorems~\ref{thm:gcd} and \ref{thm:universal} in the connected case.
In Section~\ref{sec:automorphisms} we prove Theorem~\ref{thm:matrix_group}. Here $U^*$ will make a further appearance.
In Section~\ref{sec:general_case} we deduce Theorems~\ref{thm:gcd} and \ref{thm:universal} by reduction to the connected case.
Finally, the algorithmic Theorem~\ref{thm:algos} is proved in Section~\ref{sec:algo}.

\section{Modules and Decompositions} \label{sec:Modules}

In this section we gather some results on modules, by which we mean left modules.
Let \(R\) be a ring and \(M\) an \(R\)-module.
We call an \(R\)-module \(D\) a {\em divisor\/} of \(M\) if \(M\cong D\oplus N\) for some \(R\)-module~\(N\).
If $(M_i)_{i\in I}$ is a family of submodules of $M$, then there is a natural map
\(\bigoplus_{i\in I} M_i\to M\), and we write \(\bigoplus_{i\in I} M_i = M\) if this map is
an isomorphism.
Likewise, if $C$, $D$, $N$ are submodules of~$M$, we write $C\oplus D=N$ if the natural
map $C\oplus D\to M$ is injective with image~$N$.

\begin{definition}\label{def:dec_indec}
Let \(R\) be a ring and \(M\) an \(R\)-module.
A {\em decomposition\/} of \(M\) is a pair \((I,(M_i)_{i\in I})\) where \(I\) is a set and the
\(M_i\) are submodules of \(M\) such that \(\bigoplus_{i\in I} M_i=M\); in this case we also call
\((M_i)_{i\in I}\) an $I${\em-indexed decomposition\/} of~$M$.
We call \(M\) {\em indecomposable\/} if \(M\) is non-zero and there do not exist non-zero
\(R\)-modules \(D\) and \(N\) with \(M\cong D\oplus N\).
\end{definition}

We abbreviate \((M_i)_{i\in I}\) to \((M_i)_i\) when the index set is understood.

\begin{definition}\label{def:id_dec}
For a ring \(R\) and an \(R\)-module \(M\) we define
\[\Id(R)=\{e\in R:e^2=e\},\] 
\[\Dec(M)=\{(D,N):D,\,N\text{ are submodules of }M\text{ with }D\oplus N=M\}.\]
We equip \(\Id(R)\) with a partial order given by \(e\leq f\) if and only if \(ef=fe=e\).
We equip \(\Dec(M)\) with a partial order given by \((D,N)\leq(D',N')\) if and only if there exists a
submodule \(C\subset M\) such that \(D=D'\oplus C\) and \(N\oplus C=N'\).
\end{definition}

The following 
result, which we will use to prove Theorem~\ref{thm:swop_grp},
is easily verified.

\begin{proposition}\label{prop:id_end_dec}
Let \(R\) be a ring and \(M\) an \(R\)-module.
Then we have a bijection \(\Phi\colon\Id(\End(M))\to \Dec(M)\) given by \(e\mapsto(\ker(e),\im(e))\).
If we let \(\Aut(M)\) act naturally both on \(\Id(\End(M))\) by conjugation and on \(\Dec(M)\) coordinate-wise, then \(\Phi\) is an isomorphism of partially ordered sets that respects the action of \(\Aut(M)\). 
\end{proposition}

A module has \emph{finite length}\ if every totally ordered set of submodules is finite, or equivalently if it is both Noetherian and Artinian.

\begin{theorem}[Krull--Remak--Schmidt; see Theorem~X.7.5 of \cite{lang2005algebra}]\label{thm:krs}
Suppose \(R\) is a ring and \(M\) is an \(R\)-module of finite length.
Then there exists a decomposition 
of \(M\) into finitely many indecomposable submodules, and such a decomposition is unique up to automorphisms of \(M\) and relabeling of the indices. 
\end{theorem}

\begin{remark}\label{rem:gcd}
Let $R$ be a ring, and let $\mathcal{M}$ be a non-empty set of \(R\)-modules of finite length.
As a consequence of Theorem~\ref{thm:krs}, there exists up to isomorphism exactly one \(R\)-module \(D\)
that is a divisor of
every 
$M\in\mathcal{M}$ such that
every 
$R$-module that is a divisor of
every 
$M\in\mathcal{M}$ is a divisor of~$D$; 
it is called a {\it greatest common divisor\/} of the set~$\mathcal{M}$. 
Every such $D$ is of finite length.
We say \(R\)-modules \(M\) and \(N\) are {\it coprime\/} if the greatest common divisor of \(\{M,N\}\) is \(0\).
Likewise, if $\mathcal{M}$ is a finite set of \(R\)-modules of finite length, then there exists
up to isomorphism exactly one \(R\)-module \(L\) of which each $M\in\mathcal{M}$ is a divisor
such that $L$ is a divisor of each $R$-module of finite length of which each $M\in\mathcal{M}$ is
a divisor; it is called a {\it least common multiple\/}
of~$\mathcal{M}$.
Every such $L$ is of finite length.
\end{remark}

\begin{definition}
Suppose \(R\) is a ring, \(M\) is an \(R\)-module, and \(h\in\End(M)\). We define the \(R\)-modules
\[ \Im(h) = \bigcap_{n=1}^\infty  h^n(M) \quad\text{and}\quad \Ker(h) = \bigcup_{n=1}^\infty\ker(h^n).\]
\end{definition}

\begin{lemma}[Fitting; see Theorem~X.7.3 of \cite{lang2005algebra}]\label{lem:fitting}
Suppose \(R\) is a ring, \(M\) is an \(R\)-module of finite length, and \(h\in\End(M)\).
Then \(M=\Im(h)\oplus\Ker(h)\), the restriction of \(h\) to \(\Im(h)\) is an automorphism, and 
the restriction 
of $h$ 
to \(\Ker(h)\) is nilpotent.
\end{lemma}

\begin{lemma}\label{lem:limit_iso}
Suppose \(R\) is a ring, \(M\) and \(N\) are \(R\)-modules, and \(f\colon M\to N\) and \(g\colon N\to M\) are morphisms.
Then \(f\) restricts to morphisms \(i\colon \Im(gf)\to\Im(fg)\) and \(k\colon \Ker(gf)\to\Ker(fg)\).
If \(M\) and \(N\) have finite length, then \(i\) is an isomorphism.
\end{lemma}
\begin{proof}
We have \(f(gf)^{n}(M) = (fg)^nf(M) \subset (fg)^n(N)\) for all \(n\geq 1\). 
Hence \(f(\Im(gf))\subset \Im(fg)\), 
so 
\(i\) is well-defined.
As 
$$f(\ker((gf)^{n+1}))\subset \ker(g(fg)^n)\subset\ker((fg)^{n+1})$$ 
for all \(n\geq 1\) we also get \(f(\Ker(gf))\subset\Ker(fg)\), so \(k\) is well-defined.
By symmetry we obtain a restriction \(j\colon \Im(fg)\to\Im(gf)\) of \(g\).
Under the finite length assumption both \(ji\) and \(ij\) are automorphisms by Lemma~\ref{lem:fitting}, hence \(i\) is an isomorphism.
\end{proof}

\begin{proposition}\label{prop:swap_prop}
Suppose \(R\) is a ring, \(M\) is an \(R\)-module of finite length, and \(A_1\), \(A_2\), \(B_1\),
\(B_2\subset M\) are submodules such that \(A_1\) and \(A_2\) are coprime,
\(A_1\oplus A_2=B_1\oplus B_2=M\), and \(A_1\cong B_1\). Then 
\(A_1\oplus B_2=B_1\oplus A_2=M\).
\end{proposition}

Note that under the above assumptions it immediately follows that \(A_1\oplus B_2 \cong B_1 \oplus B_2 = M\). This is not equivalent to \(A_1\oplus B_2 = M\), as this concerns a specific map \(A_1\oplus B_2\to M\). We need to show that the natural map \(B_1\to M \to A_1\) is an isomorphism.

\begin{proof}
From Theorem~\ref{thm:krs} it follows that \(A_2\cong B_2\) and thus \(B_1\) and \(B_2\) are coprime as well. By symmetry it therefore suffices to show \(B_1\oplus A_2 = M\).
We consider the maps as in the following commutative diagram, where \(\varphi\colon A_1\to B_1\) is an isomorphism, the maps to and from \(M\) are the natural inclusions and projections, and the \(f_i\) and \(g_i\) are defined to make the diagram commute.
\begin{center}
\begin{tikzcd}[row sep=small,column sep=small]
& [25pt]& A_1 &[20pt] A_1 \arrow[swap,d,outer sep=2pt,"e_1"] \arrow[rrd,bend left=16,"g_1"] &  & [25pt] \\
A_1 \arrow{r}{\varphi} \arrow[rru,bend left=16,"f_1"]\arrow[swap,rrd,bend right=16,"f_2"]& B_1 \arrow[r,"e"] & M \arrow[swap,u,outer sep=2pt,"p_1"]\arrow[d,outer sep=2pt,"p_2"] & M \arrow{r}{p} & B_1 \arrow{r}{\varphi^{-1}} &  A_1 \\
&& A_2 & A_2 \arrow[u,outer sep=2pt,"e_2"] \arrow[rru,swap,bend right=16,"g_2"]
\end{tikzcd}
\end{center}
Note that \(\id_{B_1}=pe\) and \(\id_M=e_1p_1+e_2p_2\), so 
\[\id_{A_1}=\varphi^{-1}pe\varphi=\varphi^{-1}p(e_1p_1+e_2p_2)e\varphi=\varphi^{-1}pe_1\cdot
	p_1e\varphi+\varphi^{-1}pe_2\cdot p_2e\varphi=g_1f_1+g_2f_2.\]
Lemma~\ref{lem:limit_iso} shows that \(D=\Im(g_2f_2)\cong\Im(f_2g_2)\), so \(D\) is a divisor of both \(A_1\) and \(A_2\) 
by Lemma~\ref{lem:fitting}.
Since \(A_1\) and \(A_2\) are coprime, we must have that \(D=0\) and thus \(g_2f_2\) is nilpotent.
We conclude that \(g_1f_1=\id_{A_1}-g_2f_2\) is an automorphism of \(A_1\). 
Hence $f_1$ is injective, and since \(A_1\) is of finite length it must be an automorphism. 
It follows that \(p_1 e=f_1\varphi^{-1}\colon B_1\to A_1\) is an isomorphism, 
so 
\(M=B_1\oplus A_2\), as was to be shown.
\end{proof}

\begin{definition}\label{def:dec}
Let \(R\) be a ring.
A class \(\mathcal{S}\) of \(R\)-modules is {\em multiplicative\/} if \(0\in\mathcal{S}\) and
for all \(R\)-modules \(M\), \(N\) and \(D\) with \(M\cong N\oplus D\) and \(N\), \(D\in\mathcal{S}\)
one has \(M\in\mathcal{S}\).
We say a multiplicative class \(\mathcal{S}\) of \(R\)-modules is {\em saturated\/} if for all \(M\in\mathcal{S}\) and
all divisors \(D\) of \(M\) one has \(D\in\mathcal{S}\).
For a multiplicative class \(\mathcal{S}\) of \(R\)-modules and an \(R\)-module \(M\),
write \(\textup{Dec}_\mathcal{S}(M)=\{(M_1,M_2)\in\textup{Dec}(M):M_2 \in \mathcal{S}\}\), where \(\textup{Dec}(M)\) is as in Definition~\ref{def:id_dec}. We equip \(\textup{Dec}_\mathcal{S}(M)\) with the partial order inherited from \(\textup{Dec}(M)\), and write \(\max(\textup{Dec}_\mathcal{S}(M))\) for its set of maximal elements.
\end{definition}

\begin{proposition}\label{prop:decI}
Let \(R\) be a ring, let \(\mathcal{S}\) be a multiplicative class of \(R\)-modules, and let \(M\) and \(N\) be \(R\)-modules. Then
\begin{enumerate}[nosep]
\item if \(M\cong N\) and \(N\in\mathcal{S}\), then \(M\in\mathcal{S}\);
\item the set \(\textup{Dec}_\mathcal{S}(M)\) is non-empty.
\end{enumerate}
Suppose in addition that \(\mathcal{S}\) is saturated and that \(M\) is of finite length, and let
\((A_1,A_2)\), \((B_1,B_2)\in\textup{Dec}_\mathcal{S}(M)\). Then
\begin{enumerate}[resume,nosep]
\item one has \((A_1,A_2)\in\max(\textup{Dec}_\mathcal{S}(M))\) if and only if \(0\) is the only divisor of \(A_1\) that is in \(\mathcal{S}\);
\item the set \(\max(\textup{Dec}_\mathcal{S}(M))\) is non-empty and consists of 
one 
orbit 
of \(\textup{Dec}_\mathcal{S}(M)\) 
under the action of \(\Aut(M)\);
\item if \((A_1,A_2)\), \((B_1,B_2)\in\max(\textup{Dec}_\mathcal{S}(M))\),
then \((A_1,B_2)\), \((B_1,A_2)\in\max(\textup{Dec}_\mathcal{S}(M))\).
\end{enumerate}
\end{proposition}
\begin{proof}
(i) Apply the definition 
of multiplicative
with \(D=0\). 

(ii) 
The trivial element \((M,0)\) is in \(\textup{Dec}_\mathcal{S}(M)\). 

(iii) If \((A_1,A_2)\) is maximal but \(A_1=D\oplus B_1\) for some \(D\in\mathcal{S}\) and some \(B_1\),
then \((A_1,A_2) \leq (B_1,A_2\oplus D)\in\textup{Dec}_\mathcal{S}(M)\) and thus \(A_2=A_2\oplus D\)
and \(D=0\). Conversely, suppose \(0\) is the only divisor of \(A_1\) that is in \(\mathcal{S}\)
and \((A_1,A_2)\leq (B_1,B_2)\). Then there is some \(C\) such that \(A_1=B_1\oplus C\) and \(B_2=A_2\oplus C\).
Since \(\mathcal{S}\) is saturated we have \(C\in\mathcal{S}\), and since \(C\) is a divisor
of \(A_1\) we must have that \(C=0\). Hence \((A_1,A_2)=(B_1,B_2)\) is maximal. 

(iv) Let $M=\bigoplus_{i\in I}M_i$ with each $M_i$ indecomposable. If $A_2$, respectively $A_1$, is the direct sum
of those $M_i$ that are, respectively are not, in $\mathcal{S}$,
then \((A_1,A_2)\) is in \(\textup{Dec}_\mathcal{S}(M)\) and it is maximal by (iii).
If \((B_1,B_2)\) is also maximal, then \(B_2\), respectively \(B_1\),
is a direct sum of indecomposables that are, respectively are not, in $\mathcal{S}$; this follows
from the definition of \(\textup{Dec}_\mathcal{S}(M)\) and from (iii). Since together these
decompositions give a decomposition of \(M\) into indecomposables, Theorem~\ref{thm:krs} implies that 
\(A_1\cong B_1\) and \(A_2\cong B_2\), so 
\((B_1,B_2)\) belongs to the \(\Aut(M)\)-orbit of
\((A_1,A_2)\). Because the action of \(\Aut(M)\) preserves the partial order, this orbit
is conversely contained in \(\max(\textup{Dec}_\mathcal{S}(M))\).

(v) By (iii) we have that \(A_1\) and \(A_2\) are coprime and by (iv) we have \(A_1\cong B_1\) and
\(A_2\cong B_2\).
We may conclude from Proposition~\ref{prop:swap_prop} that \((A_1,B_2),(B_1,A_2)\in\textup{Dec}_\mathcal{S}(M)\). 
Applying (iii) again we may conclude they are maximal.
\end{proof}

\section{Morphisms as modules} \label{sec:decomp_new}

In this section we will interpret a morphism of (finite) abelian groups as a (finite length) module, as expressed by Proposition~\ref{prop:matrix_cat}.
We will then study decompositions of this module and what this decomposition corresponds to in terms of the original morphism. 
This will enable us in the next section to apply the Krull--Remak--Schmidt theorem to morphisms of finite abelian groups.

We write \(\ltm\) for the ring of lower-triangular \(2\times 2\) matrices with integer coefficients, \(\catt{Ab}\) for the category of abelian groups, and \(\catt{ab}\) for the category of finite abelian groups.

\begin{definition}
Let \(\mathcal{C}\) be a category.
We define the category of {\em \(\mathcal{C}\)-morphisms}, written \homcat{\mathcal{C}}, where the objects are the morphisms of \(\mathcal{C}\) and for objects \(f\colon A\to B\) and \(g\colon C\to D\) the morphisms from \(f\) to \(g\) are the pairs \((\alpha,\beta)\in\Hom_\mathcal{C}(A,C)\times\Hom_\mathcal{C}(B,D)\) such that \(\beta f = g\alpha\), as in the following diagram:
\[\begin{tikzcd} &[-30pt] A \ar{r}{f}\ar[swap]{d}{\alpha} & B \ar{d}{\beta} &[-30pt] \\ 
& C \ar{r}{g} & D\textup{.\!} & \end{tikzcd}\]
The composition of \(\mathcal{C}\)-morphisms \((\gamma,\delta):g\to h\) and \((\alpha,\beta):f\to g\) is \((\gamma\alpha,\delta\beta)\). 
\end{definition}

\begin{remark}\label{rem:Endd-module}
Let \(d\colon A\to B\) be a morphism of abelian groups. 
Then the set \(\End(d)\subset \End(A)\times\End(B)\) of endomorphisms of \(d\) is a ring. 
Moreover, we have natural maps \(\End(d)\to\End(A)\) and \(\End(d)\to\End(B)\), turning \(A\) and \(B\) into \(\End(d)\)-modules.
Similarly we write \(\Aut(d)\) for the group of automorphisms of \(d\), which equals \(\End(d)\cap(\Aut(A)\times\Aut(B))\).
\end{remark}

The following proposition can be thought of as an explicit instance of Mitchell's embedding theorem for abelian categories.

\begin{proposition}\label{prop:matrix_cat}
There is an equivalence of categories, specified in the proof, between
the category \(\homcat{\catt{Ab}}\) and the category \ltm\catt{-Mod} of \ltm-modules.
This equivalence restricts to an equivalence between the subcategory \(\homcat{\catt{ab}}\) and the subcategory of \ltm-modules of finite length.
\begin{proof}
We will define functors \(F\colon\homcat{\catt{Ab}}\to\ltm\catt{-Mod}\) and \(G\colon\ltm\catt{-Mod}\to\homcat{\catt{Ab}}\) such that \(FG\) and \(GF\) are naturally isomorphic to the identity functors of their respective categories.
For an object \(f\colon A\to B\) we take \(F(f)\) to be \(A\oplus B\), where the \ltm-module structure is given by 
\begin{align*}
 \hspace{15em} \begin{pmatrix} x & 0 \\ y & z \end{pmatrix} \begin{pmatrix} a \\ b \end{pmatrix} = \begin{pmatrix} xa \\ yf(a)+zb \end{pmatrix} \quad\quad (\text{for }x,y,z\in\Z,\, a\in A,\, b\in B).
\end{align*}
For a \ltm-module \(M\) we take \(G(M)\) to be the morphism \(E_{11} M\to E_{22}M\) given by multiplication with \(E_{21}\), where \(E_{ij}\) is the $2 \times 2$ matrix having a \(1\) at position \((i,j)\) and zeros elsewhere. The remainder of this proposition is a straightforward verification.
\end{proof}
\end{proposition}

\begin{definition}\label{def:I}
Write \(\mathcal{I}\) for the class of \(\ltm\)-modules that correspond to isomorphisms under the equivalence of categories of Proposition~\ref{prop:matrix_cat}.
\end{definition}

One readily checks that the class \(\mathcal{I}\) is
multiplicative and saturated in the sense of Definition~\ref{def:dec}.
We observe that a \(\ltm\)-module \(M\) belongs to \(\mathcal{I}\) if
and only if its \(\ltm\)-module structure can be extended to a
\(\allm\)-module structure. This fact will not be needed, and we omit
the proof.

\begin{remark}\label{rem:transposing}
Using the equivalence of categories of Proposition \ref{prop:matrix_cat}, one can 
translate 
terminology
related to modules into terminology about morphisms of abelian groups. We briefly go
through what is most relevant to us: 
\begin{enumerate}
\item
If \(f\colon A\to B\) is a morphism of abelian groups, then a
{\em submodule\/} of the \(\ltm\)-module corresponding to \(f\) corresponds to a {\em restriction\/} of \(f\),
i.e., a morphism \(f'\colon A'\to B'\) where \(A'\subset A\) and \(B'\subset B\) are subgroups and
\(f'(a')=f(a')\in B'\) for all \(a'\in A'\).
\item
For morphisms \(f\colon A\to B\) and \(g\colon C\to D\) of abelian groups and 
for 
\(r=(\alpha,\beta)\in\Hom(f,g)\),
the image \(\im(r)\) equals the restriction \(\im(\alpha)\to\im(\beta)\) of \(g\), and the kernel \(\ker(r)\)
equals the restriction \(\ker(\alpha)\to\ker(\beta)\)  of \(f\). 
\item
If \((f_i)_{i\in I}\) is a family of morphisms \(f_i\colon A_i\to B_i\) of abelian groups, then we write \(\bigoplus_{i\in I} f_i\) for the natural map \(\bigoplus_{i\in I}A_i\to \bigoplus_{i\in I} B_i\)
and we write $f/f_i$ for the induced map  $A/A_i\to B/B_i$.
One verifies that \(\bigoplus_{i\in I} f_i\) corresponds to the direct sum of the \(\ltm\)-modules that the
\(f_i\) correspond to. 
If \(f\colon A\to B\) 
is a morphism and
\(f_i\colon A_i\to B_i\) is a family of restrictions of $f$ then, 
just as we do for modules, we will write \(\bigoplus_{i\in I} f_i = f\) 
if the natural map \(\bigoplus_{i\in I} f_i \to f\) is an isomorphism.
\item For a morphism \(f\colon A\to B\), the set $\textup{Dec}(f)$ is the set of all pairs $(f_0,f_1)$ of restrictions of $f$ such that $f_0\oplus f_1=f$, which is a partially ordered set as in Definition \ref{def:id_dec}. 
The set \(\textup{Dec}_\mathcal{I}(f)\) is the set of $(f_0,f_1)\in\textup{Dec}(f)$ such that $f_1$ is an isomorphism.
\end{enumerate}
\end{remark}

\begin{definition}\label{def:prj}
For a morphism \(f\colon A\to B\) of abelian groups we denote the ring of
morphisms from $f$ to $f$ in the category \(\homcat{\catt{Ab}}\) by
\(\End(f)\),
we define 
\[\textup{Prj}(f)=\{e\in\Id(\End(f)):\im(e)\text{ is an isomorphism}\},\] 
and we equip \(\textup{Prj}(f)\) with the partial order inherited from the partial order on
\(\Id(\End(f))\) from Definition~\ref{def:id_dec}.
\end{definition}

We have the following corollary to Proposition~\ref{prop:id_end_dec}.

\begin{corollary}\label{cor:S1_to_S2}
Let \(f\colon A\to B\) be a morphism of abelian groups.
Then we have an isomorphism \(\textup{Prj}(f)\to \textup{Dec}_\mathcal{I}(f)\) of partially ordered
sets that respects the action of \(\Aut(f)\).
\end{corollary}

\section{The group \texorpdfstring{\(U^*\)}{U*}} \label{sec:Ustar}

In this section we fix a morphism \(d\colon A\to B\) of abelian groups.
We will define a group \(U^*\) that acts on \(d\) and study some of its properties.

\begin{definition}\label{def:U}
For \(f,g\in\Hom(B,A)\), define \(f\star g= fdg \in \Hom(B,A)\), and extend \(\star\) to a ring multiplication on the additive group \(Q=Q(d)=\Z\oplus\Hom(B,A)\) by 
$$(m,f)\star (n,g)= (mn,mg+nf+fdg)$$ 
for $m,n\in\Z$ and \(f,g\in\Hom(B,A)\).
We define the multiplicative monoid 
$$U=U(d)=1+\Hom(B,A) \subset \Z\oplus\Hom(B,A) = Q$$ 
and write
\(U^*=U^*(d)=U\cap Q^*\) for the intersection of $U$ with the group of units of \(Q\).
\end{definition}

It is easy to check that $Q$ is indeed a ring
with unit element $1=(1,0)$, 
and that the projection map $Q\to\Z$
is a ring homomorphism with kernel \(\Hom(B,A)\). The inverse image of $1$ equals $U$, and
\(U^*\) is a group because it is the kernel of the induced group homomorphism \(Q^*\to\Z^*\).
The following lemma is
easy 
to verify.

\begin{lemma}\label{lem:ring_hom}
We have a ring homomorphism \(q\colon Q\to\End(d)\) defined by sending $1$ to the identity \(\id_d\)
and \(f\in\Hom(B,A)\) to \((fd,df)\).
It restricts to a group homomorphism \(U^*\to\Aut(d)\).
\end{lemma}

\begin{remark}\label{rem:Qmodule}
Recall from Remark~\ref{rem:Endd-module} that \(A\) and \(B\) are \(\End(d)\)-modules.
The map $q$ makes $A$ and $B$ into $Q$-modules in such a way that $d$ is $Q$-linear.
\end{remark}

\begin{definition}\label{def:S} 
We write 
$$\textup{Id}_0=\textup{Id}_0(d)=\Id(Q)\cap \Hom(B,A)=\{f\in\Hom(B,A):fdf=f\}$$ 
with \(Q\)
as in Definition~\ref{def:U}.
We equip \(\textup{Id}_0\) with the partial order inherited from the partial order on \(\Id(Q)\) as in
Definition~\ref{def:id_dec}.
\end{definition}

\begin{remark}
From Lemma~\ref{lem:ring_hom} we get a map \(U^*\to\Aut(d)\), i.e., \(U^*\) acts on \(d\). In turn, since
an isomorphism between $d$'s induces a bijection between their $\textup{Id}_0$'s, the group \(\Aut(d)\)
acts on \(\textup{Id}_0\). However, \(U^*\) acts directly on \(\textup{Id}_0\) by conjugation within \(Q\).
Although we will not need it, both of the induced maps \(U^*\to\Aut(\textup{Id}_0)\) are the same. 
\end{remark}

Recall the terminology $\im(r)$, $\ker(r)$, and \(\bigoplus_{i\in I} f_i\) from Remark~\ref{rem:transposing}, and \textup{Prj}(d) from Definition~\ref{def:prj}. If 
$e \in\End(d)$, we write $e=(e_A,e_B)$ with $e_A\in\End(A)$ and $e_B\in\End(B)$.

\begin{proposition}\label{prop:S_to_S1}
The map \(q\colon Q\to\End(d)\) from Lemma~\ref{lem:ring_hom} restricts to an
isomorphism 
\(\textup{Id}_0(d)\isom \textup{Prj}(d)\)
of partially ordered
sets that respects the action of \(\Aut(d)\). 
Its inverse is the map given by \(e\mapsto 0_e \oplus \im(e)^{-1} \),
where 
\(0_e\colon \ker(e_B)\to \ker(e_A)\)
is the zero map. 
\end{proposition}
\begin{proof}
Write \(\Phi\colon\textup{Id}_0(d)\to\End(d)\) for the restriction of \(q\), and write \(\Psi\colon\textup{Prj}(d)\to\Hom(B,A)\) for the map \(e\mapsto 0_e \oplus \im(e)^{-1}\). To see that \(\Psi\) is well defined, note that \(e\) and hence \(e_A\) and \(e_B\) are idempotents, so  \(A=\ker(e_A)\oplus \im(e_A)\) and likewise for \(B\) by Proposition~\ref{prop:id_end_dec}.

Let \(e\in \textup{Prj}(d)\) and \(f=\Psi(e)\).
Note that \(fd=e_A\), by considering the restriction to \(\im(e_A)\) and \(\ker(e_A)\)
separately, and similarly \(df=e_B\) and \(fe_B=f\).
Hence \(q(f)=e\) and \(q\circ\Psi=\id_{\textup{Prj}(d)}\).
For \(g\in \Hom(B,A)\)  
we have 
\(g\star f = gdf= g e_B\). 
In particular \(f\star f=fe_B=f\), so \(f\in \textup{Id}_0\).
Thus \(\Psi\) restricts to \(\textup{Prj}(d)\to \textup{Id}_0\).

Since \(q\) is a ring homomorphism, it 
maps idempotents to idempotents.
For \(f\in \textup{Id}_0\) it follows from \(fdf=f\) that \(f\) and \(d\) are mutually inverse
when restricted to \(\im(df)\to\im(fd)\), respectively \(\im(fd)\to\im(df)\). 
In particular \(\im(q(f))\), which is precisely the restriction of \(d\) just mentioned, is an
isomorphism. 
Hence \(\Phi\) restricts to \(\textup{Id}_0\to\textup{Prj}(d)\).
Moreover, the restriction of \(f\) to \(\ker(df)\to\ker(fd)\) is zero since $fdf=f$, so
\(\Psi(\Phi(f))=f\). We conclude that \(\Phi\) and \(\Psi\) are mutually inverse.

All constructions are functorial in \(d\) and thus \(\Aut(d)\) commutes with \(\Phi\).
The definition of the partial order on idempotents is completely algebraic, so the partial
order is preserved by \(\Phi\), which is the restriction of a ring homomorphism.
If 
\(e,e'\in\textup{Prj}(d)\) are such that \(ee'=e=e'e\), 
then
\(\Psi(e)\star\Psi(e')=\Psi(e)e'_B=\Psi(e)e_Be'_B=\Psi(e)e_B=\Psi(e)\), and likewise
\(\Psi(e')\star\Psi(e)=\Psi(e)\), so \(\Psi\) preserves the partial order as well.
\end{proof}

In the following results, we use the terminology from Remark \ref{rem:transposing}.
We let $U^*$ act on \(\textup{Dec}(d)\) via the map \(U^*\to\Aut(d)\) from Lemma~\ref{lem:ring_hom}.

\begin{lemma}\label{lem:orbit_lemma}
Let \(d_i\colon A_i\to B_i\) with \(i\in\{-1,0,1\}\) be restrictions of $d$ such that
\((d_0,d_1)\) and \((d_0,d_{-1})\) belong to \(\textup{Dec}(d)\), and suppose that
\(d_0\) or \(d_1\) is an isomorphism. Then 
\((d_0,d_{-1})\in U^*\cdot (d_0,d_1)\).
\begin{proof}
We have \(A_0\oplus A_1 \cong A \cong A_0 \oplus A_{-1}\).
Hence the map $A_1 \to A_{-1}$ given by $x \mapsto x_{-1}$ where
$x = x_0 + x_{-1}$ with $x_i\in A_i$ is an isomorphism. Similarly,
we have a natural isomorphism \(g_1\colon d_1\to d/d_0\to d_{-1}\), and its extension 
\(g=\id_{d_0}\oplus \mathop{g_1} \in \Aut(d)\) maps \((d_0,d_1)\) to \((d_0,d_{-1})\). 
Letting 
\(r=\id_d-g
\in\End(d)$, then
\(r(d_1)\subset d_0\) and \(r(d_0) = 0\), so 
\(r^2=0\).
We first construct $f\in\Hom(B,A)$ that maps to \(r\) under \(q\colon Q\to\End(d)\).
Write \(r=(r_A,r_B)\) with $r_A \in \End(A)$ and $r_B \in \End(B)$.
Since \(d_0\) or \(d_1\) is invertible, there exists \(f_1\colon B_1\to A_0\) such that the diagram
\begin{center}
\begin{tikzcd}
A_1 \ar{r}{d_1} \ar[swap]{d}{r_A} & B_1 \ar{d}{r_B} \ar[swap,dotted]{dl}{f_1}\\
A_0 \ar[swap]{r}{d_0} & B_0
\end{tikzcd}
\end{center}
commutes. Then \(f=0\oplus f_1\) with \(0\colon B_0\to A_1\) satisfies \((fd,df)=r\), so
$f$ 
does map to \(r\) under \(q\colon Q\to\End(d)\). From \(f\star f \star f=fdfdf=r_A^2f=0\) we
see that $f$ is nilpotent, so the element \(1-f\in U\) belongs to $U^*$. 
Since
\(1-f\)
maps to \(\id_d-r=g\)
via $q$, 
it sends \((d_0,d_1)\) to \((d_0,d_{-1})\).
\end{proof}
\end{lemma}

The proof of the following proposition, which can be considered a sharpening of Proposition~\ref{prop:decI}.iv when \(R=\ltm\), is the main reason for considering \(d\) as a module.

\begin{proposition}\label{prop:thm_proof_in_S2}
Assume \(A\) and \(B\) are finite. Then the set of maximal elements of
\(\textup{Dec}_\mathcal{I}(d)\) 
equals one 
orbit of \(\textup{Dec}_\mathcal{I}(d)\) under the action of \(U^*\).
\end{proposition}
\begin{proof}
By Proposition~\ref{prop:matrix_cat} we may apply Proposition~\ref{prop:decI}.iv.
Thus, it suffices to show that any two maximal elements
\((d_0,d_1),(e_0,e_1)\in\textup{Dec}_\mathcal{I}(d)\) are in the same \(U^*\)-orbit.
Recall that \((d_0,e_1)\in\textup{Dec}_\mathcal{I}(d)\) by Proposition~\ref{prop:decI}.v.
Applying Lemma~\ref{lem:orbit_lemma} we obtain \((d_0,e_1)\in U^*\cdot (d_0,d_1)\)
since $d_1$ is an isomorphism, and 
\((e_0,e_1)\in U^*\cdot(d_0,e_1)\)
since $e_1$ is an isomorphism.
Thus 
\((e_0,e_1)\in U^*\cdot(d_0,e_1)=U^*\cdot (d_0,d_1)\).
\end{proof}

\section{Graded rings}\label{sec:GroupRings}

In this section we consider gradings, which may be viewed as a generalization of group rings. 
In Section~\ref{sec:degree_map} we will use them to prove Theorem~\ref{thm:swop_grp}.
We begin by giving the definitions as we need them, and state some results from \cite{UniversalGradings}.

\begin{definition}[Definition~1.1 and Lemma~7.1 in \cite{UniversalGradings}]\label{def:grading}
Let $R$ be a commutative ring. A {\em grading} of \(R\) is a pair \((\Delta,\mathcal{R})\) where \(\Delta\)
is a (multiplicatively written) abelian group and \(\mathcal{R}=(R_\delta)_{\delta\in\Delta}\) is a
\(\Delta\)-indexed decomposition of \(R\) as a \(\Z\)-module, as defined in 
Definition~\ref{def:dec_indec}, 
such that \(R_\gamma R_\delta \subset R_{\gamma\delta}\) for all \(\gamma\), \(\delta\in\Delta\).
For a grading \((\Delta,\mathcal{R})\) with \(\mathcal{R}=(R_\delta)_{\delta\in\Delta}\) and a group homomorphism
\(f\colon\Delta\to\textup{E}\), we define \(f_*(\mathcal{R})\) to be the \(\textup{E}\)-indexed decomposition
\((S_\epsilon)_{\epsilon\in\textup{E}}\) of $R$ defined by \(S_\epsilon=\sum_{\delta\in f^{-1}\epsilon} R_\delta\); then
\((\textup{E},f_*(\mathcal{R}))\) is a grading of~\(R\).
We turn the class of gradings of \(R\) into a category by defining a morphism
\((\Delta,\mathcal{R})\to(\textup{E},\mathcal{S})\) to be a group homomorphism
\(f\colon\Delta\to\textup{E}\) for which $f_*(\mathcal{R})=\mathcal{S}$.
\end{definition}

From now on, we will usually write $(R_\delta)_{\delta}$ instead of $(R_\delta)_{\delta\in\Delta}$.

\begin{remark}\label{rem:group_ring_to_grading}
Note that every non-zero commutative group ring \(A[G]\) naturally comes with a grading
\((G,(Ag)_{g})\). Let \(R\) be a commutative ring.
Analogously to writing \(\bigoplus_{i\in I} M_i = M\) for submodules \(M_i\) of some \(R\)-module \(M\) when the natural map \(\bigoplus_{i\in I} M_i\to M\) is an isomorphism, we write \(R=A[G]\) for a subring \(A\) of \(R\) and subgroup \(G\) of \(R^*\) if the natural map \(A[G]\to R\) is an isomorphism.
\end{remark}

\begin{definition}[Definition~1.2 in \cite{UniversalGradings}]\label{def:univgrading}
Let $R$ be a commutative ring. A grading \((\Gamma,\mathcal{R})\) of $R$ is called {\em universal\/}
if it is an initial object of the category of gradings of $R$, or equivalently if for every grading
\((\textup{E},\mathcal{S})\) of $R$ there is a unique group homomorphism $f\colon\Gamma\to\textup{E}$
such that $f_*(\mathcal{R})=\mathcal{S}$. If a universal grading \((\Gamma,\mathcal{R})\) exists, then
it is unique up to a unique isomorphism and we refer to it as {\em the\/} universal grading of~$R$.
We write $\Gamma(R)=\Gamma$ for the group of this grading.
\end{definition}

\begin{remark}\label{rem:functoriality}
If $R$ and $R'$ are commutative rings that have universal gradings, then any ring isomorphism $R\to R'$
induces a group isomorphism $\Gamma(R)\to\Gamma(R')$, so $\Gamma(R)$ behaves functorially under ring
isomorphisms; in particular, the group $\Aut(R)$ of ring automorphisms of $R$ acts in a natural way on
$\Gamma(R)$.
\end{remark}

Two important results on gradings of orders from \cite{UniversalGradings} are fundamental to the present paper.
The first concerns the existence of a universal grading.
Recall a commutative ring is {\em reduced\/} if it has no non-zero nilpotent elements.

\begin{theorem}[Theorem~1.3 of \cite{UniversalGradings}]
\label{thm:universal_main_thm1}
If \(R\) is a reduced order, then \(R\) has a universal grading and \(\Gamma(R)\) is finite.
\end{theorem}

The second result relates to roots of unity.
Recall that a commutative ring \(R\) is {\em connected\/} if it has exactly two idempotents and that \(\mu(R)\) is the group of roots of unity in $R$.

\begin{theorem}[Theorem~1.5.iii in \cite{UniversalGradings}]
\label{thm:universal_main_thm2}
If \(R\) is a connected order and \((\Delta,(R_\delta)_{\delta})\) is a grading of \(R\), then \(\mu(R)\subset\bigcup_{\delta\in\Delta}R_\delta\).
\end{theorem}

The proofs of these two theorems as given in \cite{UniversalGradings} are of a number-theoretic nature. 
For algebraic arguments one may refer to \cite{Daan,Daan2}.

Another useful fact is that the properties of being reduced and being connected are preserved under construction of group rings.
We write \(\textup{nil}(R)\) for the set of nilpotent elements of a commutative ring \(R\).

\begin{proposition}[Theorem~1.5 in \cite{UniversalGradings}] 
\label{prop:reduced_connected_groupring}
Let \(A\) be an order and \(G\) a finite abelian group. Then:
\begin{enumerate}[nosep]
\item \(\textup{nil}(A[G])=\textup{nil}(A)[G]\), and \(A\) is reduced if and only if \(A[G]\) is reduced;
\item \(\textup{Id}(A[G])=\textup{Id}(A)\), and \(A\) is connected if and only if \(A[G]\) is connected;
\item If \(A\) is connected, then \(\mu(A[G])=\mu(A)\times G\).
\end{enumerate}
\end{proposition}
\begin{proof}
We apply the theory of gradings to the natural grading of \(A[G]\).
For (i) we apply Theorem~1.5.i in \cite{UniversalGradings} so that \(\textup{nil}(A[G])=\bigoplus_{g\in G} (\textup{nil}(A[G])\cap Ag) = \textup{nil}(A)[G]\).
The remaining equivalence follows trivially.
Theorem~1.5.ii and iii in \cite{UniversalGradings} prove (ii) and (iii).
\end{proof}

\begin{proposition}\label{prop:universal_efficient}
Suppose \(\overline{R}=(\Delta,(R_\delta)_\delta)\) is a grading of a commutative ring \(R\) and let 
$$\Delta'=\langle \delta\in\Delta : R_\delta \neq 0 \rangle.$$ 
Then:
\begin{enumerate}[nosep]
\item We have that \(\overline{R}{}'=(\Delta',(R_\delta)_\delta)\) is a grading of \(R\). 
\item The inclusion \(i\colon \Delta'\to\Delta\) is a morphism \(\overline{R}{}'\to\overline{R}\) of gradings.
\item If \(\overline{S}\) is a grading of $R$ and there exists a 
morphism \(f\colon\overline{R}\to\overline{S}\), then 
there exists a unique morphism \(f'\colon\overline{R}{}'\to\overline{S}\). It equals \(f\circ i\). 
\item If there exists a morphism from \(\overline{R}{}'\) to a universal grading, then \(\overline{R}{}'\) is universal.
\item If \(\overline{R}\) is universal, then $\Delta = \Delta'$.
\end{enumerate}
\end{proposition}
\begin{proof}
Both (i) and (ii) are trivial. For (iii), clearly \(f\circ i\) is such a morphism. For uniqueness, 
it follows from the definitions that 
\(f'\) must equal \(f\) for all \(\delta\in\Delta\) such that \(R_\delta \neq 0\), and such \(\delta\) generate \(\Delta'\). For (iv), we 
have a map from \(\overline{R}{}'\) to any other grading by passing through the universal grading, and such 
a 
map is unique by (iii).
For (v), if \(\overline{R}\) is universal, then so is \(\overline{R}{}'\) by (ii) and (iv), and then $i$ is a bijection since universal objects are uniquely unique.
\end{proof}

\begin{proposition}\label{prop:product_grading}
Let \(S\) and \(T\) be orders, write \(R=S\times T\), and let \(\pi\colon R\to S\) be the
natural projection. 
\begin{enumerate}[nosep]
\item If \((\Delta,(R_\delta)_\delta)\) is a grading of \(R\), then
\((\Delta',(\pi(R_\delta))_\delta)\) with \(\Delta'=\langle\delta\in\Delta:\pi(R_\delta)\neq
0\rangle\) is a grading of \(S\), and if the former is universal, then so is the latter.
\item If \((\textup{E},(S_\epsilon)_\epsilon)\) and \((\textup{Z},(T_\zeta)_\zeta)\) are
gradings of \(S\), respectively \(T\), then
\((\textup{E}\times\textup{Z},(R_{(\epsilon,\zeta)})_{(\epsilon,\zeta)})\) with
$R_{1,1}=S_1\times T_1$, $R_{\epsilon,1}=S_\epsilon\times\{0\}$ for $\epsilon\ne1$,
$R_{1,\zeta}=\{0\}\times T_\zeta$ for $\zeta\ne1$, and $R_{(\epsilon,\zeta)}=\{(0,0)\}$ for
$\epsilon\ne1$, $\zeta\ne1$,
is a grading of \(R\). 
If the 
first 
two are universal, then so is the latter.
\end{enumerate}
\end{proposition}
\begin{proof}
(i) Clearly 
\(S=\sum_{\delta\in\Delta'}\pi(R_\delta)\). We identify \(S\) with
\(S\times\{0\}\subset R\), so that  
\(\pi(R_\delta)=(1,0)\cdot R_\delta\). As
\((1,0)\in R_1\) by Theorem~1.5.ii in \cite{UniversalGradings}, we find
\(\pi(R_\delta)\subset R_\delta\), and hence the sum of the \(\pi(R_\delta)\) is a
direct sum. It follows that \((\Delta',(\pi(R_\delta))_\delta)\) is a grading of \(S\).

Suppose that \((\Delta,(R_\delta)_\delta)\) is universal and let
\((\textup{E},(S_\epsilon')_\epsilon)\) be a grading of~\(S\). 
Then \((\textup{E},(S_\epsilon)_\epsilon)\) with \(S_1=S_1'\times T\) and
\(S_\epsilon=S_\epsilon'\times\{0\}\) for \(\epsilon\neq 1\) is a grading of \(R\).
By universality, there is a group homomorphism $f\colon\Delta\to\textup{E}$ such that
$f_*((R_\delta)_\delta)=(S_\epsilon)_\epsilon$. Its restriction $f'$ to a map
\(\Delta'\to\textup{E}\) is a morphism \((\Delta',(\pi(R_\delta))_\delta)\to
(\textup{E},(S_\epsilon')_\epsilon)\), and the map is unique by Proposition~\ref{prop:universal_efficient}.iii.
We conclude that \((\Delta',(\pi(R_\delta))_\delta)\) is universal.

(ii) The first statement is immediate. For universality, any grading $(\Delta,\mathcal{R})$
of $R$ induces, as in (i), gradings of $S$ and of $T$ by the same group $\Delta$, which
come from unique group homomorphisms $\textup{E}\to\Delta$ 
and 
$\textup{Z}\to\Delta$. 
One
readily checks that the induced group homomorphism $\textup{E}\times\textup{Z}\to\Delta$ is
the unique morphism $(\textup{E}\times\textup{Z},(R_{(\epsilon,\zeta)})_{(\epsilon,\zeta)})\to
(\Delta,\mathcal{R})$. The details are left to the reader.
\end{proof}

\begin{definition}\label{def:S0}
When \(R\) is a commutative ring, we define the set 
\[\GpRg(R)=\{ (A,G):A\subset R\text{ a subring, } G\subset R^*\text{ a subgroup, } A[G] = R \},\]
and equip it with a partial order \(\leq\) given by \((B,H)\leq (A,G)\) if and only if \(H\subset G\) and \(B\supset A\).
\end{definition}

Note that \(\Aut(R)\) naturally acts component-wise on \(\GpRg(R)\).

\begin{lemma}\label{lem:max_exists}
Suppose \(R\) is a non-zero order. Then for each \((A,G)\in\GpRg(R)\) the order of \(G\) is at most the rank of \(R\) as a \(\Z\)-module, and \(\GpRg(R)\) contains a maximal element.
\end{lemma}
\begin{proof}
By definition of \(\GpRg(R)\) the elements of \(G\) are linearly independent, from which the first claim follows.
We have 
\((R,1)\in\GpRg(R)\), 
so 
\(\GpRg(R)\) is 
not empty. 
Thus if \((A,G)\in\GpRg(R)\) and \(\#G\) is maximal, then \((A,G)\) is a maximal element of \(\GpRg(R)\).
\end{proof}

\begin{lemma}\label{lem:connected_factorization}
Let \(R\) be a connected order and let \((A,G)\), \((B,H)\in \GpRg(R)\) be such that \((B,H)\leq (A,G)\). Then with \(J=G\cap \mu(B)\) we have \(G=J\times H\) and \(B=A[J]\).
\end{lemma}
\begin{proof}
By Lemma~\ref{lem:max_exists} the group \(H\) is finite, and by Proposition~\ref{prop:reduced_connected_groupring}.iii the multiplication map \(\mu(B)\times H\to\mu(R)\) is an isomorphism. 
Since the inverse image of \(G\) is \(J\times H\), 
we have \(G=J\times H\).
Thus \(A[J][H]=A[J\times H]=A[G]=B[H]\) and therefore \(A[J]=B\).
\end{proof}

\begin{example}
The conclusion to Lemma~\ref{lem:connected_factorization} does not hold in general for non-connected orders.
Let \(p\) be prime and let \(G=C_p\times C_p\) with \(C_p\) a group of order \(p\). 
Then \(G\) is a $2$-dimensional \(\F_p\)-vector space and thus there are precisely \(p+1\) subgroups \(H_0,\dotsc, H_p\) of \(G\) of order \(p\). 
We have \(H_i \cdot H_j = G\) if and only if \(i\neq j\). 
Let \(R=\Z[G]\times \Z[G]\) and let \(\Delta\colon G\to\mu(R)\) be the map given by \(g\mapsto(g,g)\).
Now consider the elements 
$$(\Z\times\Z,\Delta(G))\geq(\Z[H_0]\times \Z[H_1],\Delta(H_p))$$ 
of \(\GpRg(R)\).
As Proposition~\ref{prop:reduced_connected_groupring}.iii implies 
$$\mu(\Z[H_0]\times \Z[H_1])=\mu(\Z[H_0])\times\mu(\Z[H_1])=\{(\pm h_0, \pm h_1):h_0\in H_0,\, h_1\in H_1\},$$ 
we get \(J=\Delta(G)\cap\mu(\Z[H_0]\times \Z[H_1])=1\) and \((\Z\times \Z)[J]\neq \Z[H_0]\times \Z[H_1]\). 
\end{example}

Recall that we say a commutative ring \(R\) is {\em stark\/} if there do not exist a ring \(A\) and a non-trivial group \(G\) such that \(R\) is isomorphic to the group ring \(A[G]\), or equivalently for \(R\) non-zero, if \(\#\GpRg(R)=1\).

\begin{lemma}\label{lem:stark_iff_maximal}
Let \(R\) be a non-zero commutative ring and let \((A,G)\in \GpRg(R)\).
If \((A,G)\) is maximal, then \(A\) is stark.
When \(R\) is a connected order, the converse also holds.
\end{lemma}
\begin{proof}
If \(A=B[J]\) for some \(J\subset A^*\), then \((A,G)\leq (B,J\times G)\in \GpRg(R)\). Hence if \((A,G)\) is maximal we have \((A,G)=(B,J\times G)\) and thus \(J=1\), so \(A\) is stark. For connected orders, the converse follows from Lemma~\ref{lem:connected_factorization}.
\end{proof}

Note that from Theorem~\ref{thm:gcd} it follows that maximality of \((A,G)\in\GpRg(R)\) for a non-zero reduced order \(R\) is equivalent to \(A\) being stark even when \(R\) is not connected.
However, we have not proved this yet.

\section{The degree map}\label{sec:degree_map}

In this section we extract from a connected reduced order \(R\) a morphism 
$d$ 
of abelian groups. We will describe \(\GpRg(R)\) in terms of \(d\) using the theory 
in 
Section~\ref{sec:Ustar} 
and then prove Theorem~\ref{thm:swop_grp}.

\begin{lemma}\label{lem:degree_abhom}
Let \(R\) be a connected reduced order and let \((\Gamma(R),(R_\gamma)_\gamma)\) be its universal grading (see Definition~\ref{def:univgrading}). 
Then there exists a morphism of finite abelian groups \(d\colon\mu(R)\to\Gamma(R)\) that sends \(\zeta\in\mu(R)\) to the unique \(\gamma\in\Gamma(R)\) such that \(\zeta\in R_\gamma\).
\end{lemma}
\begin{proof}
The group \(\Gamma(R)\) is finite by Theorem~\ref{thm:universal_main_thm1}, and \(\mu(R)\) is finite by Lemma~3.3.ii in \cite{RootsOfUnity}. 
By Theorem~\ref{thm:universal_main_thm2}, if \(\zeta\in\mu(R)\), then there exists a \(\gamma\in\Gamma(R)\) such that \(\zeta\in R_\gamma\). 
The element \(\gamma\) is unique, since \(R_\gamma\cap R_\delta = 0\) for all \(\gamma\neq \delta\).
That \(d\) is a homomorphism follows from the definitions.
\end{proof}

\begin{definition}\label{def:degree-map}
For a connected reduced order \(R\) we call the map \(d=d_R\colon\mu(R)\to\Gamma(R)\) from Lemma~\ref{lem:degree_abhom} the {\em degree map} of \(R\).
\end{definition}

The above definition depends on the choice of universal grading.
However, the universal grading of \(R\) is uniquely unique.
Moreover, the proof of Theorem~1.3 of \cite{UniversalGradings}, which states that a reduced order has a universal grading, exhibited an explicit canonical choice of universal grading.
Thus we can confidently refer to {\em the} degree map of a connected reduced order.
We now describe the degree map \(d_{A[G]}\) of \(A[G]\).

\begin{proposition}\label{prop:universal_of_groupring}
Let \(A\) be a connected reduced order and let \(G\) be a finite abelian group.
Let \((\Gamma(A),(A_\gamma)_{\gamma})\) and \((\Gamma(A[G]),(R_\gamma)_\gamma)\) be the universal grading of \(A\) and \(A[G]\), respectively. Then
\begin{enumerate}[nosep]
\item we have \(\Gamma(A[G])=\Gamma(A)\times G\), and \(R_{(\gamma,g)}=A_\gamma\cdot g\) for all \(\gamma\in\Gamma(A)\) and \(g\in G\);
\item if we identify \(\mu(A[G])\) with \(\mu(A)\times G\) as in Proposition~\ref{prop:reduced_connected_groupring}.iii, then the degree map 
$$d_{A[G]}\colon\mu(A)\times G \to\Gamma(A)\times G$$ 
equals \(d_A\times\id_G\);
\item we have \(\Gamma(A)=\langle \gamma\in\Gamma(A[G]) : R_\gamma \cap A \neq 0 \rangle\).
\end{enumerate}
\end{proposition}
\begin{proof}
Let \(\mathcal{A}=(\Gamma(A),(A_\gamma)_\gamma)\) and \(\mathcal{R}=(\Gamma(A[G]),(R_\gamma)_\gamma)\) be the universal gradings of \(A\) and \(A[G]\) respectively and
define \(\mathcal{A}[G]=(\Gamma(A)\times G, (A_\gamma\cdot g)_{(\gamma,g)})\).
By universality there exists a unique morphism of gradings \(\varphi\colon \mathcal{R}\to \mathcal{A}[G]\), which by Definition~\ref{def:grading} is a group homomorphism \(\Gamma(A[G])\to \Gamma(A)\times G\), and we will show that it is an isomorphism.
Let \(\pi\colon\Gamma(A)\times G\to G\) be the projection and \(\Delta = \ker(\pi \varphi)\). 
For \(g\in G\) we have \(g\in R_{d_{A[G]}(g)}\) and \(g\in A_1 \cdot g\), so \(\pi\varphi d_{A[G]}\) is the identity on \(G\).
It follows that \(\Gamma(A[G])=\Delta \times G\).
Then \(\mathcal{R}_A=(\Delta,(R_\delta)_\delta)\) is a grading of \(A\), and \(\varphi\) restricts to a morphism of gradings \(\varphi'\colon\mathcal{R}_A\to \mathcal{A}\) with \(\varphi=\varphi'\times \id_G\).
With \(\Delta'=\langle \delta\in\Delta : R_\delta\neq 0\rangle\) we have
\[ \bigoplus_{(\delta,g)\in\Delta'\times G} R_{\delta} \cdot g = A[G],\]
so by Proposition~\ref{prop:universal_efficient}.v we obtain \(\Delta'\times G = \Gamma(A[G]) = \Delta \times G\).
Hence \(\Delta'=\Delta\), so \(\mathcal{R}_A\) is universal by Proposition~\ref{prop:universal_efficient}.iv.
It follows that \(\varphi'\) and hence \(\varphi\) is an isomorphism, proving (i).
Now (ii) and (iii) follow by inspection.
\end{proof}

Proposition~\ref{prop:universal_of_groupring}.ii expresses the degree map of \(A[G]\) in terms of \(G\) and the degree map of \(A\), but we will mainly use it in the opposite direction.
Specifically, for a connected reduced order \(R\), an element \((A,G)\in\GpRg(R)\) corresponds to a certain decomposition \((d_A,\id_G)\in\textup{Dec}_\mathcal{I}(d)\) of the degree map \(d\) of \(R\), as defined in Definition~\ref{def:dec} and Definition~\ref{def:I}.
In Theorem~\ref{thm:composite_maps} 
we will show
it is in fact a bijective correspondence.
This together with Proposition~\ref{prop:thm_proof_in_S2} will prove Theorem~\ref{thm:swop_grp}.

\begin{remark}\label{rem:r2d2}
Let \(R\) be a connected reduced order with universal grading \((\Gamma,(R_\gamma)_\gamma)\) and
degree map \(d\colon \mu\to\Gamma\).
Note that the group \(\Aut(R)\) acts on the category of gradings of \(R\).
Under this action, \(\sigma\in\Aut(R)\) sends \((\Gamma,(R_\gamma)_\gamma)\) to \((\Gamma,(\sigma(R_\gamma))_\gamma)\), which is again a universal grading of \(R\).
Thus by universality, \(\sigma\) induces a unique isomorphism 
$\Gamma\isom\Gamma$. 
It follows that \(\Aut(R)\) acts on \(\Gamma\).
Clearly \(\Aut(R)\) acts on \(\mu\), and it is easy to see the following diagram commutes
\begin{center}
\begin{tikzcd}
\mu\vphantom{\Gamma} \ar{r}{d} \ar[swap]{d}{\sigma} & \Gamma\vphantom{\mu} \ar{d}{\sigma} \\
\mu\vphantom{\Gamma} \ar{r}{d} & \Gamma,
\end{tikzcd}
\end{center}
i.e., we obtain an action \(\Aut(R)\to\Aut(d)\) of \(\Aut(R)\) on \(d\).
\end{remark}

For a degree map \(d\colon\mu\to\Gamma\) we define \(U^*(d)\) as in Definition~\ref{def:U}.

\begin{lemma}\label{lem:group_actions}
Let \(R\) be a connected reduced order with universal grading \((\Gamma,(R_\gamma)_{\gamma\in\Gamma})\) and corresponding degree map \(d\colon \mu\to\Gamma\).
Let \(\varphi\colon U^*(d)\to\Aut(d)\) be as in Lemma~\ref{lem:ring_hom} and \(\chi\colon \Aut(R)\to\Aut(d)\) as in Remark~\ref{rem:r2d2}. We then have a commutative diagram
\begin{center}
\begin{tikzcd}[column sep=tiny]
&U^*(d) \arrow[swap]{dl}{\psi}\arrow{dr}{\varphi}& \\
\Aut(R)\arrow{rr}{\chi} & & \Aut(d) 
\end{tikzcd}
\end{center}
where \(\psi\) is a morphism given by \(1+f\mapsto (x\in R_\gamma\mapsto f(\gamma)\cdot x)\).
\end{lemma}
\begin{proof}
Let \(1+f,1+g\in U^*\) and recall 
that 
their product equals \((1+f)\star(1+g)=1+f+g+fdg\) in \(U^*\). 
It is easy to see that \(\psi(1+f)\) is an endomorphism of \(R\). 
For \(\gamma\in \Gamma\) we have
\begin{gather*}
x \in R_\gamma \xmapsto{\psi(1+g)}  g(\gamma)\cdot x \in R_{dg(\gamma)} \cdot R_\gamma \subset R_{dg(\gamma)\cdot\gamma}\\
\xmapsto{\psi(1+f)} 
f( dg(\gamma)\cdot \gamma ) \cdot g(\gamma)\cdot x = f(\gamma)g(\gamma)fdg(\gamma) \cdot x,
\end{gather*}
so indeed \(\psi(1+f)\circ\psi(1+g)=\psi((1+f)\star(1+g))\). It follows that \(\psi(1+f)\in\Aut(R)\) and that \(\psi\) is a morphism.

Let \(1+f\in U^*\)  and write \(F=\psi(1+f)\).
For \(\zeta\in\mu\) we have \(F(\zeta)=f(d\zeta)\zeta\), so \(F|_{\mu(R)}=\id_\mu+fd\).
For \(\gamma\in\Gamma\) and \(x\in R_\gamma\) non-zero we have \(F(x)=f(\gamma)\cdot x\), so the induced action on \(\Gamma\) sends \(\gamma\) to \(df(\gamma)\gamma\).
Hence \(1+f\) gets sent to \(\id_\Gamma+df\), since \(\{\gamma\in\Gamma:R_\gamma\neq 0\}\) is a generating set of \(\Gamma\) by Proposition~\ref{prop:universal_efficient}.v.
We conclude that \(\chi(\psi(1+f))=(\id_{\mu}+fd,\id_\Gamma+df)=\varphi(1+f)\), as was to be shown.
\end{proof}

\begin{example}
The map \(\psi\colon U^*\to\Aut(R)\) need not be injective, even when \(R\) is stark.
Consider the subring \(R=\Z\cdot(1,1)+2 S\) of \(S=\Z[\i]\times\Z[\i]\) where \(\i^2=-1\), which is clearly connected, reduced, and has \(\mu(R)=\{\pm 1\}\times\{\pm1\}\).
Let \(\Gamma=\mu(R)\) and write 
$$R_{1,1}=R\cap (\Q\times\Q)=\Z\cdot (1,1)+\Z\cdot (1,-1),$$ 
$$R_{1,-1}=2\i\cdot(\Z\times \{0\}), \quad 
R_{-1,1}=2\i\cdot(\{0\}\times\Z), \quad R_{-1,-1}=0.$$
Then \((\Gamma,(R_\gamma)_{\gamma})\) is the universal grading of \(R\).
Consider the identity \(\id\colon \Gamma\to\mu\).
Note that \(2\id=0\) and \(d=0\), hence \((1+\id)^2=1\) in \(Q\) and so \(1+\id\in U^*\).
Moreover, \(\psi(1+\id)\)
is the identity of \(R\), so 
$\psi$ is not injective.
To see \(R\) is stark, we
can apply 
Lemma~\ref{lem:eq_stark}.vi 
below 
since \(d=0\).
\end{example}

For a degree map \(d\colon\mu\to\Gamma\) we write \(\textup{Id}_0(d)\) for the partially ordered set \(\{ f\in\Hom(\Gamma,\mu):fdf=f \}\) corresponding to \(d\), as defined in Definition~\ref{def:S}.

\begin{theorem}\label{thm:S0_to_S}
Let \(R\) be a connected reduced order with universal grading \((\Gamma,(R_\gamma)_{\gamma})\) and corresponding degree map \(d\colon\mu\to\Gamma\).
The map \(\textup{Id}_0(d)\to \GpRg(R)\) given by \(f\mapsto( \bigoplus_{\gamma\in\ker(f)} R_\gamma, \im(f) )\) is a well-defined isomorphism of partially ordered sets and respects the action of \(\Aut(R)\) 
of 
Remark~\ref{rem:r2d2}. Its inverse is the map given by \((A,G)\mapsto 0_{\Gamma(A)\to\mu(A)} \oplus d_{G\to d(G)}^{-1}\).
\end{theorem}
\begin{proof}
Write \(\Phi\) for the map \(\textup{Id}_0(d)\to \GpRg(R)\), which we first verify is well-defined. 
Suppose \(f\in \textup{Id}_0(d)\). 
As \(\ker(f)\) is a subgroup of \(\Gamma\) we have that \(B=\bigoplus_{\gamma\in\ker(f)} R_\gamma\) is a subring of \(R\).
Furthermore, from \(fdf=f\) it follows that the restriction of \(d\) to \(\im(f)\to\im(df)\)  is an isomorphism, with the restriction of \(f\) to \(\im(df)\to\im(f)\) as inverse, and 
\(\Gamma=\ker(f)\oplus\im(df)\).
Thus
\begin{align*}
R=\bigoplus_{\gamma\in\Gamma} R_\gamma =\bigoplus_{\zeta\in \im(f)} \Big(\bigoplus_{\gamma\in \ker(f) d(\zeta)  }R_\gamma \Big) = \bigoplus_{\zeta\in \im(f)} B\zeta = B[\im(f)].
\end{align*}
Hence \(\Phi(f)\in \GpRg(R)\) and \(\Phi\) is well-defined. 

We 
next 
construct an inverse \(\Psi\) of \(\Phi\). 
Suppose \((A,G)\in \GpRg(R)\). 
By Proposition~\ref{prop:universal_of_groupring} we may factor \(\Gamma(R)\) and \(\mu(R)\) such that \(d\) is given by \(d\colon\mu(A)\times G \to \Gamma(A) \times d(G)\), where 
$$\Gamma(A)=\langle \gamma\in\Gamma(R) : R_\gamma \cap A \neq 0 \rangle,$$ 
and the restriction \(d_1\) of \(d\) to \(G\to d(G)\) is an isomorphism.
Now define \(f\colon\Gamma(R)\to\mu(R)\) as the product of \(0\colon\Gamma(A)\to\mu(A)\) and \(d_1^{-1}\).
Clearly \(fdf=f\), so \(\Psi\) is well-defined.

For \((A,G)\in\GpRg(R)\) we have \(\ker(\Psi(A,G))=\Gamma(A)\) and \(\im(\Psi(A,G))=G\), so indeed \(\Phi\circ\Psi=\id\).
Conversely, let \(f\in\textup{Id}_0(d)\) and \(f'=\Psi(\Phi(f))\). 
Restricted to \(\im(df)\to\im(f)\), both \(f\) and \(f'\) equal the inverse of a restriction of \(d\), and are \(0\) on a subgroup complementary to \(\im(df)\).
By Proposition~\ref{prop:universal_of_groupring}.iii we have
\[ \ker(f') = \Big\langle \gamma'\in\Gamma(R) : R_{\gamma'} \cap \bigoplus_{\gamma\in\ker(f)} R_\gamma \neq 0 \Big\rangle = \langle \gamma\in\ker(f) : R_\gamma \neq 0 \rangle \subset\ker(f),\]
so equality must hold and \(f=f'\).
Thus \(\Psi\circ\Phi=\id\), and \(\Phi\) and \(\Psi\) are bijections.

It follows 
easily 
from the definitions
that \(\Phi\) and \(\Psi\) respect the partial order 
and that the action of \(\Aut(R)\) commutes with \(\Phi\).
\end{proof}

\begin{remark}\label{rem:mui_direct_summand}
Recall that \(U^*(d)\) was defined in Definition~\ref{def:U}, and that
it acts on $R$ and hence on \(\GpRg(R)\) by Lemma~\ref{lem:group_actions}. Moreover, \(U^*(d)\) acts on \(d\) and hence on 
\[\textup{Dec}_\mathcal{I}(d)=\{ (d_0,d_1) : d_0 \oplus d_1 = d \textup{ and } d_1 \textup{ is an isomorphism} \}\]
as in Remark~\ref{rem:transposing}.iv. 
For \((d_0,d_1)\in\textup{Dec}_\mathcal{I}(d)\) we will write \(\mu_i\) and \(\Gamma_i\) for the finite abelian groups such that \(d_i\colon\mu_i\to\Gamma_i\). 
Note that if \(d_0\oplus d_1=d\), then  
\(\mu_0\oplus \mu_1=\mu\) and \(\Gamma_0\oplus\Gamma_1=\Gamma\).
\end{remark}

\begin{theorem}\label{thm:composite_maps}
Let \(R\) be a connected reduced order with degree map \(d\) and universal grading \((\Gamma,(R_\gamma)_{\gamma})\).
We have an isomorphism of partially ordered sets \(\textup{Dec}_\mathcal{I}(d)\isom \GpRg(R)\) given by
\begin{align*} 
\big(d_0\colon\mu_0\to\Gamma_0,\, d_1\colon\mu_1\to\Gamma_1) &\mapsto \Big( \bigoplus_{\gamma\in \Gamma_0} R_\gamma, \mu_1 \Big), \\
\big(d|_{\mu(A)\to\Gamma(A)},\, d|_{G\to d(G)}\big) &\mapsfrom \big( A, G \big),
\end{align*}
that respects the action of \(U^*(d)\), where \(\Gamma(A)=\langle\gamma\in\Gamma : 0\neq R_\gamma\subset A\rangle\).
\end{theorem}

Before we prove Theorem~\ref{thm:composite_maps}, we remark that the notation \(\Gamma(A)\) used in the theorem was already reserved for the group of the universal grading of \(A\). 
However, by Proposition~\ref{prop:universal_of_groupring}.iii there is no ambiguity, as they are uniquely isomorphic.
In fact, the map \(d|_{\mu(A)\to\Gamma(A)}\) is the degree map of \(A\).

\begin{proof}
Theorem~\ref{thm:S0_to_S}, Proposition~\ref{prop:S_to_S1}, and Corollary~\ref{cor:S1_to_S2} give explicit isomorphisms of partially ordered sets \(\GpRg(R)\to \textup{Id}_0(d)\to \textup{Prj}(d)\to \textup{Dec}_\mathcal{I}(d)\), from which one readily reads off that their composition and its inverse are as given in the theorem.
From Lemma~\ref{lem:group_actions} it follows that the isomorphisms respect the action of \(U^*(d)\).
\end{proof}

\begin{theorem}\label{thm:strong_swop_grp}
Let \(R\) be a connected reduced order with degree map $d$, and suppose \((A,G),(B,H)\in\GpRg(R)\) are such that \(A\) and \(B\) are stark. Then 
$A\cong B$
as rings, $G\cong H$ as groups, and \((A,G)\), \((A,H)\), \((B,G)\) and \((B,H)\) are all in
the same \(U^*(d)\)-orbit of \(\GpRg(R)\).
\end{theorem}
\begin{proof}
Let \(\Phi\colon\textup{Dec}_\mathcal{I}(d)\to\GpRg(R)\) be the isomorphism of Theorem~\ref{thm:composite_maps}.
Suppose \((A,G),(B,H)\in \GpRg(R)\) are such that \(A\) and \(B\) are stark.
Then \((A,G)\) and \((B,H)\) are maximal elements of \(\GpRg(R)\) by Lemma~\ref{lem:stark_iff_maximal}, and thus \(\Phi(A,G)=(d_0,d_1)\) and \(\Phi(B,H)=(e_0,e_1)\) are maximal in \(\textup{Dec}_\mathcal{I}(d)\).
Then by Proposition~\ref{prop:decI}.v and Proposition~\ref{prop:thm_proof_in_S2} all of \((d_0,d_1)\), \((d_0,e_1)\), \((e_0,d_1)\), and \((e_0,e_1)\) are maximal and in the same \(U^*\)-orbit.
Since  
\(\Phi(d_0,e_1)=(A,H)\) and \(\Phi(e_0,d_1)=(B,G)\), and \(\Phi\) respects the action of \(U^*\),
the last assertion of the theorem follows. As a consequence, \((A,G)\) and \((B,H)\) are in the same
orbit of $\Aut(R)$, so 
$A\cong B$ as rings and $G\cong H$ as groups.
\end{proof}

Theorem~\ref{thm:swop_grp} 
is an immediate consequence of Theorem~\ref{thm:strong_swop_grp}.
Note that in the connected case Theorem~\ref{thm:universal} also follows from Theorem~\ref{thm:strong_swop_grp}. In Section~\ref{sec:general_case} we shall see that the same applies for Theorem~\ref{thm:gcd}.

\begin{example}\label{ex:list_all_of_gprg}
Let \(\langle\sigma\rangle\) be a group of order $2$ and let \(R=\Z[\i][\langle\sigma\rangle]\), where \(\i^2=-1\).
We will compute \(\GpRg(R)\).
By Proposition~\ref{prop:reduced_connected_groupring}.i,ii the ring \(R\) is both reduced and connected. 
With $\Gamma = (\Z/2\Z)^2$, 
consider the grading 
\((\Gamma,(R_{a,b})_{(a,b)})\) 
of \(R\) with \(R_{a,b}=\Z\i^a\sigma^b\), where although \(\i^a\) is not well-defined, \(\Z\i^a\) is.
Since a universal grading exists, and all \(R_{a,b}\) are of rank 1 over \(\Z\), this must be the universal grading.
Let 
\(d\colon\mu\to \Gamma\) 
be the degree map. 
It follows from Proposition~\ref{prop:reduced_connected_groupring}.iii that \(\mu = \langle \i,\sigma\rangle \cong \Z/4\Z \times \Z/2\Z\).
We will first compute \(\textup{Dec}_\mathcal{I}(d)\).

Suppose we have \((d_0,d_1)\in \textup{Dec}_\mathcal{I}(d)\) with \(d_i\colon\mu_i\to\Gamma_i\) as in Remark~\ref{rem:mui_direct_summand}.
If \(\mu_1=1\), then \(d_0=d\), and \((d_0,d_1)\) corresponds
via Theorem~\ref{thm:composite_maps} 
to the trivial element \((R,1)\) of \(\GpRg(R)\).
Now suppose \(\mu_1\neq 1\).
Since \(d_1\) is an isomorphism, the groups \(\mu_1\) and \(\Gamma_1\) are isomorphic, so \(\mu_1\) is isomorphic to a direct factor of \(\mu\) and of $\Gamma$. Since $\Z/2\Z$ is the greatest common divisor of \(\mu\) and $\Gamma$ as $\Z$-modules (in the sense of Remark~\ref{rem:gcd}), we have that \(\mu_1\) is a direct factor of \(\mu\) isomorphic to \(\Z/2\Z\).
It follows that 
\(\mu_1=\langle(-1)^b\sigma\rangle\) for some \(b\in\Z/2\Z\), and the corresponding group \(\Gamma_1\) equals \(\langle (0,1)\rangle\) in both cases.
On the other hand \(\mu_0=\langle \i \sigma^a \rangle\) for some \(a\in(\Z/2\Z)\) 
since 
it must be a cyclic group of order 4, and \(\Gamma_0=\langle(1,a)\rangle\).
Upon inspection, all pairs \((a,b)\) do indeed give a decomposition \((d_0,d_1)\in \textup{Dec}_\mathcal{I}(d)\).
The rings corresponding to the possible \(d_0\) are \(\Z[\i\sigma^a]\), and the groups corresponding to \(d_1\) are \(\langle(-1)^b\sigma\rangle\).
This gives
\[\GpRg(R)=\{(R,1)\}\cup\{ (\Z[\i\sigma^a],\langle(-1)^b\sigma\rangle):a,b\in\Z/2\Z \}.\]
Interesting to note is that, although \((\Z[\i\sigma],\langle\sigma\rangle)\) differs from \((\Z[\i\sigma],\langle-\sigma\rangle)\), the corresponding gradings are isomorphic, since \(\Z[\i\sigma]\cdot \sigma = \Z[\i\sigma] \cdot (-\sigma)\).
\end{example}

\begin{example}\label{ex:counter_connected_swop}
The conclusion to Theorem~\ref{thm:swop_grp} does not hold in general for non-connected reduced orders.
Let \(C\) be a non-trivial finite abelian group and consider \(R = \Z[C\times C]\times \Z[C] \). Let
\begin{align*} 
A &= \Z[C\times 1]\times \Z, \quad\quad G=\{( (1,\gamma),\gamma ):\gamma\in C \}, \\
B &= \Z[1\times C]\times \Z, \quad\quad H=\{( (\gamma,1),\gamma ):\gamma\in C \}.
\end{align*}
Then \(A\) and \(B\) are stark, and \(A[G]=R=B[H]\).
However, the natural map \(A[H]\to R\) has image \(\Z[C\times 1]\times\Z[C]\neq R\).
\end{example}

\section{Automorphisms of group rings}\label{sec:automorphisms}

In this section we will describe \(\Aut(A[G])\), for a stark connected reduced order \(A\) with degree map \(d\) and a finite abelian group \(G\), in terms of \(U^*(d)\), \(G\), and \(\Aut(A)\). 

For a ring \(R\), write \(\textup{Jac}(R)=\{x\in R : 1+RxR \subset R^*\}\) for the {\em Jacobson radical} of \(R\).
Recall the definitions of \(\textup{Id}_0(d)\) from Definition~\ref{def:S} and 
\(\Gamma(A)\) from Definition~\ref{def:univgrading}.
In this section we write \(\textup{Id}_0(A)\) for \(\textup{Id}_0(d)\) and similarly for \(Q\), \(U\), and \(U^*\) as defined in Definition~\ref{def:U}.
In our context \(U^*(A)\) is equal to \(U(A)\) due to the following.

\begin{lemma}\label{lem:eq_stark}
Let \(A\) be a connected reduced order with degree map \(d\colon\mu\to\Gamma\).
Then the following are equivalent: 
\textup{(i)} \(A\) is stark; 
\textup{(ii)} \(\textup{Id}_0(A)=\{0\}\);
\textup{(iii)} the ideal \(\Hom(\Gamma,\mu)\subset Q(A)\) consists of only nilpotent elements; 
\textup{(iv)} \(\Hom(\Gamma,\mu)=\textup{Jac}(Q(A))\);
\textup{(v)} \(U^*(A)=U(A)\);
\textup{(vi)} for all abelian groups \(\Omega\) and morphisms \(f\colon\Omega\to\mu\) and \(g\colon\Gamma\to\Omega\), the element \(gdf\in\End(\Omega)\) is nilpotent.
\end{lemma}
\begin{proof}
We will write \(Q=Q(A)\) and similarly for \(U\) and \(U^*\). 
(i \(\Leftrightarrow\) ii) This follows from the bijection of Theorem~\ref{thm:S0_to_S}.
(ii \(\Rightarrow\) iii) Let \(f\in\Hom(\Gamma,\mu)\). 
Since the semigroup \(\Hom(\Gamma,\mu)\) with the multiplication from \(Q\) is finite by Lemma~\ref{lem:degree_abhom}, some power of \(f\) is idempotent and thus zero.
(iii \(\Rightarrow\) iv) It follows that \(1+\Hom(\Gamma,\mu)\subset Q^*\) and thus \(\Hom(\Gamma,\mu)\subset \textup{Jac}(Q)\). 
The 
surjection \(Q\twoheadrightarrow\Z\) 
must map \(\textup{Jac}(Q)\) to \(\textup{Jac}(\Z)=\{0\}\), so \(\textup{Jac}(Q)\subset \Hom(\Gamma,\mu)\).
(iv \(\Rightarrow\) v) We have \(U^*\subset U = 1+\textup{Jac}(Q)\subset U^*\).
(v \(\Rightarrow\) ii) It follows that \(1\) is the only idempotent of \(U\). The involution \(x\mapsto 1-x\) on \(Q\) shows that \(\Hom(\Gamma,\mu)\) and \(1-\Hom(\Gamma,\mu)=U\) contain equally many idempotents, hence \(\textup{Id}_0(A)=\{0\}\).
(vi \(\Rightarrow\) iii) It suffices to show that \(df\in\End(\Gamma)\) is nilpotent for all 
\(f\in\Hom(\Gamma,\mu)\), 
so take \(\Omega=\Gamma\) and \(g=\id_\Gamma\). 
(iii \(\Rightarrow\) vi) As \((gdf)^{n+1}=g(dfg)^ndf\) it holds that \(gdf\) is nilpotent if \(dfg\) is nilpotent. The latter holds because \(fg\in\Hom(\Gamma,\mu)\subset Q\) is nilpotent.
\end{proof}

A category \(\mathcal{C}\) is {\em small} if the class of objects of \(\mathcal{C}\) is a set and if for any two objects \(A\) and \(B\) of \(\mathcal{C}\) the class \(\Hom(A,B)\) is a set. 
A category \(\mathcal{C}\) is {\em preadditive} (see Section 1.2 in \cite{CategoryHandbook}) if for any two objects \(A\) and \(B\) of \(\mathcal{C}\) the class \(\Hom(A,B)\) is an abelian group such that composition of morphisms is bilinear, i.e., for all objects \(A\), \(B\), and \(C\) and morphisms \(f,f':A\to B\) and \(g,g':B\to C\) we have 
$$g\circ (f+f')=(g\circ f)+(g\circ f') \; \text{ and } \; (g+g')\circ f = (g\circ f)+(g'\circ f).$$ 

\begin{lemma}\label{lem:matrix_ring_from_category}
Let \(\mathcal{C}\) be a preadditive small category with precisely two objects \(\mathbf{0}\) and \(\mathbf{1}\). 
Then:
\begin{enumerate}
\item
With \(M_{ij}=\Hom(j,i)\) for \(i,j\in\{\mathbf{0},\mathbf{1}\}\) both \(M_{\mathbf{00}}\) and \(M_{\mathbf{11}}\) are rings and \(M_{\mathbf{01}}\) and \(M_{\mathbf{10}}\) are a \(M_{\mathbf{00}}\)-\(M_{\mathbf{11}}\)-bimodule and \(M_{\mathbf{11}}\)-\(M_{\mathbf{00}}\)-bimodule respectively. 
\item
The product of groups
\[\textup{M}(\mathcal{C})= \prod_{i,j\in\{\mathbf{0},\mathbf{1}\}} M_{ij} = \begin{pmatrix} M_{\mathbf{00}} & M_{\mathbf{01}} \\ M_{\mathbf{10}} & M_{\mathbf{11}}\end{pmatrix}\]
is a ring with 
respect to 
the addition and multiplication 
implied by the matrix notation. 
\item
If \(M_{\mathbf{01}} \cdot M_{\mathbf{10}} = \im( M_{\mathbf{01}} \tensor M_{\mathbf{10}} \to M_{\mathbf{00}} ) \subset \textup{Jac}(M_{\mathbf{00}})\), then \(M_{\mathbf{10}} \cdot M_{\mathbf{01}} \subset\textup{Jac}(M_{\mathbf{11}})\),
\[ \textup{Jac}(\textup{M}(\mathcal{C})) = \begin{pmatrix} \textup{Jac}(M_{\mathbf{00}}) & M_{\mathbf{01}} \\ M_\mathbf{10} & \textup{Jac}(M_\mathbf{11})\end{pmatrix}\text{,} \quad\text{and}\quad \textup{M}(\mathcal{C})^*=\begin{pmatrix} M_{\mathbf{00}}^* & M_{\mathbf{01}} \\ M_{\mathbf{10}} & M_{\mathbf{11}}^*\end{pmatrix}\text{.}\]
\end{enumerate}
\end{lemma}
\begin{proof}
That the \(M_{ij}\) are groups, and that the addition is compatible with the composition of morphisms, follows from the fact that \(\mathcal{C}\) is preadditive. It is then 
easy to verify 
that the \(M_{ij}\) are rings and modules as claimed, and that \(\textup{M}(\mathcal{C})\) is a ring,
giving (i) and (ii). 

Now suppose \(M_\mathbf{01} \cdot M_\mathbf{10} \subset \textup{Jac}(M_\mathbf{00})\).
We will show that for all \(m\in M_\mathbf{01}\) and \(n\in M_\mathbf{10}\) we have \(nm\in\textup{Jac}(M_\mathbf{11})\).
Let \(s\in M_\mathbf{11}\). Then \((ms)n\in M_\mathbf{01} \cdot M_\mathbf{10} \subset \textup{Jac}(M_\mathbf{00})\), so \(1+msn\) has an inverse \(r\in M_\mathbf{00}\). Then
\[ (1-snrm)(1+snm)=1-sn(r(1+msn)-1)m = 1-sn(1-1)m=1.\]
Hence \(1+snm\) has a left inverse \(1-snrm\), and 
similarly \(1-snrm\) is a right inverse of \(1+snm\).
Thus \(1+snm\in M_\mathbf{11}^*\) and \(nm\in\textup{Jac}(M_\mathbf{11})\).
We conclude that \(M_\mathbf{10} \cdot M_\mathbf{01} \subset \textup{Jac}(M_\mathbf{11})\).

Consider \(T=\big(\begin{smallmatrix} \textup{Jac}(M_\mathbf{00}) & M_\mathbf{01} \\ 0 & 0\end{smallmatrix}\big)\) and \(B = \big(\begin{smallmatrix} 0 & 0 \\ M_\mathbf{10} & \textup{Jac}(M_\mathbf{11} )\end{smallmatrix}\big)\) and write \(J=T+B\).
We will first show that \(T\subset\textup{Jac}(\textup{M}(\mathcal{C}))\).
For \(x=\big(\begin{smallmatrix} a & b \\ 0 & 0\end{smallmatrix}\big)\in T\) it suffices to show for all \(y=\big( \begin{smallmatrix} r & m \\ n & s\end{smallmatrix}\big)\in \textup{M}(\mathcal{C})\) that \(1+xy\in \textup{M}(\mathcal{C})^*\).
As \(1+xy=\big(\begin{smallmatrix} 1+ar+bn & ma+bs \\ 0 & 1\end{smallmatrix}\big)\) is upper diagonal, it is invertible if its diagonal elements are. 
The element \(1+ar+bn\) is invertible because \(ar+bn\in\textup{Jac}(M_\mathbf{00})\).
So 
\(T\subset\textup{Jac}(\textup{M}(\mathcal{C}))\).
Analogously \(B\subset\textup{Jac}(\textup{M}(\mathcal{C}))\). 
Thus we have a two-sided ideal \(J\subset\textup{Jac}(\textup{M}(\mathcal{C}))\).
To see 
equality, note that the ring 
$$\textup{M}(\mathcal{C})/ J \cong (M_\mathbf{00}/\textup{Jac}(M_\mathbf{00})) \times (M_\mathbf{11}/\textup{Jac}(M_\mathbf{11}))$$ 
has a trivial Jacobson radical.

An element of \(\textup{M}(\mathcal{C})\) is a unit if and only if it maps to a unit in \(\textup{M}(\mathcal{C})/\textup{Jac}(\textup{M}(\mathcal{C}))\), hence if and only if its diagonal elements are units, proving the final statement.
\end{proof}

Naturally, the construction \(\textup{M}(\mathcal{C})\) can be generalized to categories \(\mathcal{C}\) with any finite number of objects.
We call \(\textup{M}(\mathcal{C})\) the {\em matrix ring} of \(\mathcal{C}\). 

\begin{remark}
\label{rem:crazycats}
Given four abelian groups \(M_{ij}\) with \(i,j\in\{0,1\}\) together with compatible (i.e., associative) multiplications \(M_{ij}\tensor M_{jk} \to M_{ik}\) for all \(i,j,k\in\{0,1\}\) with appropriate unit elements, we can construct the preadditive category \(\mathcal{C}\) with two objects \(0\) and \(1\), with \(\Hom(j,i)=M_{ij}\), and 
with 
composition 
being 
these multiplications.
In particular, 
if \(M_{00}\) and \(M_{11}\) are rings, \(M_{01}\) is an \(M_{00}\)-\(M_{11}\)-bimodule, and \(M_{10}\) is an \(M_{11}\)-\(M_{00}\)-bimodule, then it remains only to specify the multiplications \(M_{01}\tensor M_{10}\to M_{00}\) and \(M_{10}\tensor M_{01}\to M_{11}\).
\end{remark}

Let \(A\) be a connected reduced order and \(G\) a finite abelian group.
Recall that \(\mu(A)\) and \(\Gamma(A)\) are \(Q(A)\)-modules by Remark~\ref{rem:Qmodule}, hence \(\Hom(G,\mu(A))\) and \(\Hom(\Gamma(A),G)\) are respectively left and right \(Q(A)\)-modules.
We 
next 
describe \(U^*(A[G])\) in terms of \(A\) and \(G\).

\begin{proposition}\label{prop:U_facts}
Let \(A\) be a connected reduced order and \(G\) a finite abelian group. 
Then:
\begin{enumerate} 
\item
We have a matrix ring 
\[E = \begin{pmatrix}Q(A) & \Hom(G,\mu(A)) \\ \Hom(\Gamma(A),G) & \End(G) \end{pmatrix}\]
where \(\Hom(G,\mu(A))\tensor \Hom(\Gamma(A), G) \to \Hom(\Gamma(A),\mu(A)) \subset Q(A)\) is the composition map and \(\Hom(\Gamma(A),G)\tensor \Hom(G,\mu(A))\to\End(G)\) is given by \(g\tensor f \mapsto gdf\).
\item
There is a natural ring isomorphism 
\(E\isom Q(A[G])\) 
that respects the action of $\Aut(A)$. 
\item
If \(A\) is stark,
then the map in (ii) 
restricts to an isomorphism
\[\begin{pmatrix}U^*(A) & \Hom(G,\mu(A)) \\ \Hom(\Gamma(A),G) & \Aut(G) \end{pmatrix} \isom U^*(A[G]).\]
\end{enumerate}
\end{proposition}
\begin{proof}
For (i), apply Remark~\ref{rem:crazycats} and Lemma~\ref{lem:matrix_ring_from_category}.ii.
Since all multiplications are defined in terms of compositions of morphisms, the associativity conditions are trivially satisfied.

Write \(\Gamma=\Gamma(A)\) and \(\mu=\mu(A)\). 
For (ii), we have by Proposition~\ref{prop:universal_of_groupring}.ii that
\[Q(A[G]) \; = \; \Z\oplus\Hom(\Gamma\times G,\mu\times G) \;  \cong \;  \Z \oplus \begin{pmatrix} \Hom(\Gamma,\mu) & \Hom(G,\mu) \\ \Hom(\Gamma,G) & \End(G) \end{pmatrix},\]
where the isomorphism is one of abelian groups. Then the map \(Q(A[G])\to E\) with respect to the latter representation given by
\[\big(n, \begin{pmatrix} p & q \\ r & s \end{pmatrix} \big) \mapsto \begin{pmatrix} (n,p) & q \\ r & n + s \end{pmatrix}\]
is an isomorphism of rings that by functoriality respects the action of \(\Aut(A)\).

For (iii), 
suppose \(A\) is stark. 
Then 
\(\Hom(\Gamma,\mu)=\textup{Jac}(Q(A))\) by Lemma~\ref{lem:eq_stark}.
It follows that 
the ideal \(\Hom(G,\mu)\cdot\Hom(\Gamma,G)\subset\Hom(\Gamma,\mu)\) is contained in \(\textup{Jac}(Q(A))\).
Now apply 
Lemma~\ref{lem:matrix_ring_from_category}.iii.
\end{proof}

In Remark~\ref{lem:Q_inclusion_group_ring} and Proposition~\ref{prop:AutR_facts} we describe \(\Aut(A[G])\) in terms of \(A\) and \(U^*(A[G])\).

\begin{remark}
\label{lem:Q_inclusion_group_ring}
Let \(G\) be a finite abelian group. Then \(-[G]\) and \(U^*\) act functorially on isomorphisms of connected reduced orders.
Let \(A\) be a connected reduced order.
From Proposition~\ref{prop:universal_of_groupring}.ii we get a natural inclusion 
$$\Hom(\Gamma(A),\mu(A))\to\Hom(\Gamma(A[G]),\mu(A[G])),$$ 
which extends to an inclusion of rings \(Q(A)\to Q(A[G])\). 
Then we have a commutative diagram
\begin{center}
\begin{tikzcd}[column sep=large]
U^*(A)    \ar[tail]{d} \ar{r}{\textup{Lem~\ref{lem:group_actions}}} & \Aut(A) \ar{r}{U^*} \ar[tail]{d}{-[G]} & \Aut(U^*(A)) \\
U^*(A[G])                  \ar{r}{\textup{Lem~\ref{lem:group_actions}}} & \Aut(A[G]) \ar{r}{U^*} & \Aut(U^*(A[G])),
\end{tikzcd}
\end{center}
and the composition \(U^*(A)\to\Aut(U^*(A))\) is the conjugation map. 
\end{remark}

\begin{proposition}\label{prop:AutR_facts}
Let \(A\) be a stark connected reduced order and \(G\) a finite abelian group.
Then the maps and actions from 
Remark~\ref{lem:Q_inclusion_group_ring} fit in an exact sequence
\[ 0\to U^*(A) \xrightarrow{\iota} U^*(A[G]) \rtimes \Aut(A) \xrightarrow{\pi} \Aut(A[G]) \to 0, \] 
where \(\iota\) and \(\pi\) are homomorphisms 
such that 
\(\iota(u) = (u^{-1},u)\) and \(\pi\) maps each component to \(\Aut(A[G])\).
\end{proposition}
\begin{proof}
For all \(u,v\in U^*(A)\) we have 
$$\iota(u)\iota(v)=(u^{-1},u)(v^{-1},v)=(u^{-1}(uv^{-1}u^{-1}),uv)=\iota(uv)$$ 
by 
Remark~\ref{lem:Q_inclusion_group_ring}, so \(\iota\) is a homomorphism. Moreover, \(\iota\) is injective because it maps injectively to the first factor. By the same remark \(\pi\) is a homomorphism.

We will now show that \(\pi\) is surjective.
Suppose \(\sigma\in\Aut(A[G])\).
By Theorem~\ref{thm:strong_swop_grp} there 
exists 
\(1+f\in U^*(A[G])\) that 
maps
 \((A,G)\) to \((\sigma(A),\sigma(G))\), so without loss of generality we may assume \(\sigma(A)=A\) and \(\sigma(G)=G\).
By applying the restriction \(\sigma|_A\in\Aut(A)\) we may assume \(\sigma\) is the identity on \(A\).
Consider the map \(f\colon\Gamma(A)\times G\to\mu(A[G])\) given by \((\delta,g)\mapsto \sigma(g)g^{-1}\) and note that \(1+f\in U(A[G])\) gets mapped to \(\sigma\).
We similarly obtain the inverse of \(1+f\) in \(U(A[G])\) from \((\delta,g)\mapsto \sigma^{-1}(g)g^{-1}\), so \(1+f\in U^*(A[G])\).
It follows that \(\sigma\) is in the image of \(\pi\) and thus \(\pi\) is surjective.

To show the sequence is exact, it remains to show \(\im(\iota)=\ker(\pi)\).
It is clear that \(\im(\iota)\subset\ker(\pi)\), so suppose \((1+f,\alpha)\in\ker(\pi)\).
As \(\alpha^{-1}\) equals the restriction of \(1+f\) by assumption, it suffices to show that \(1+f\in U^*(A)\).
For \(g\in G\) we have \(g=(1+f)\alpha(g)=f(g)g\), and multiplying by \(g^{-1}\) we obtain \(1=f(g)\), i.e., \(G\subset\ker(f)\).
Moreover 
\(\im(f)\subset \mu(A)\), since multiplication by any unit \((\zeta,g)\in\mu(A)\times G=\mu(A[G])\) not in \(\mu(A)\) sends \(A\) to \(Ag \neq A\).
Hence \(f\in\Hom(\Gamma(A),\mu(A))\) and \(1+f\in U(A)\). 
The same holds for the inverse \(1+e\in U^*(A[G])\) of \(1+f\), so \(1+e\in U(A)\) and thus \(1+f\in U^*(A)\). 
It now follows that \((1+f,\alpha)=\iota(1+e)\), so \(\ker(\pi)\subset\im(\iota)\), as was to be shown.
\end{proof}

Proposition~\ref{prop:U_facts} and Proposition~\ref{prop:AutR_facts} combined gives us a description of \(\Aut(A[G])\) in terms of \(A\) and \(G\).
We now prove Theorem~\ref{thm:matrix_group} and describe \(\Aut(A[G])\) by less canonical means.

\begin{lemma}\label{lem:Hom_action_on_Aut}
Let \(A\) be a stark connected reduced order. Then the group \(\Hom(\Gamma(A),\mu(A))\) has a (right) action on the set \(\Aut(A)\), which for \(\alpha\in\Aut(A)\) and \(f\in\Hom(\Gamma(A),\mu(A))\) is given by
\[(\alpha,f)\mapsto \alpha+f = \big( x \in A_\gamma \mapsto \alpha(x) \cdot f(\gamma) \big).\]
\end{lemma}
\begin{proof}
Let \(\alpha\in\Aut(A)\) and \(f,g\in\Hom(\Gamma(A),\mu(A))\).
Note that 
$$\alpha+f=\alpha\circ(1+\alpha^{-1}f)\in\Aut(A),$$ 
where \(1+\alpha^{-1} f \in U(A)=U^*(A)\) by Lemma~\ref{lem:eq_stark} and the composition is taken inside \(\Aut(A)\) via Lemma~\ref{lem:group_actions}. 
For \(\gamma\in\Gamma(A)\) and \(x\in A_\gamma\) we clearly have
\[ [(\alpha+f)+g](x)= [\alpha+f](x) \cdot g(\gamma) = \alpha(x) \cdot f(\gamma) \cdot g(\gamma) = [\alpha+(f+g)](x), \]
so the action is well-defined.
\end{proof}

\noindent{\bf{Proof of Theorem~\ref{thm:matrix_group}.}}
To check that \(M\) is a group it remains to verify that \(t_1ds_2+\sigma_1\sigma_2\in\Aut(G)\).
This follows from Lemma~\ref{lem:eq_stark}, namely \(t_1ds_2\in\textup{Jac}(\End(G))\). 
Note that the map 
$$\vartheta\colon M\to\Aut(A[G])$$ 
can be written as the composition of the homomorphism \(\varphi\colon M\to U^*(A[G])\rtimes\Aut(A)\) given by
\begin{align*}
\begin{pmatrix} \alpha & s \\ t & \sigma \end{pmatrix} &\mapsto \begin{pmatrix} 1 & s \\ t\alpha^{-1} & \sigma \end{pmatrix} \cdot \alpha
\end{align*}
where \(U^*(A[G])\) is written in terms of the matrix representation of Proposition~\ref{prop:U_facts}.iii, and the homomorphism \(\pi\colon U^*(A[G])\rtimes\Aut(A) \to \Aut(A[G])\) from Proposition~\ref{prop:AutR_facts}.

The map \(\pi\) is still surjective when restricted to the image of \(\varphi\).
Namely any 
$$\begin{pmatrix} u & s \\ t & \sigma \end{pmatrix}\cdot \alpha\in U^*(A[G])\rtimes\Aut(A)$$ 
has the same image as \(\big(\begin{smallmatrix} 1 & s \\ t\beta^{-1} & \sigma \end{smallmatrix}\big) \cdot \beta\alpha\), where \(\beta\) is the image of \(u\) in \(\Aut(A)\).
Hence the map \(\vartheta\) is surjective.
By Proposition~\ref{prop:U_facts}.iii and Proposition~\ref{prop:AutR_facts}, respectively, we have 
\[\frac{\#M}{\# U^*(A[G])} = \frac{\#\Aut(A)}{\# U^*(A)} = \frac{\#\Aut(A[G])}{\#U^*(A[G])},\]
so the groups \(M\) and \(\Aut(A[G])\) have the same (finite) cardinality, so \(\vartheta\) is bijective.

\section{Proofs of Theorems~\ref{thm:gcd}, \ref{thm:universal}, and \ref{thm:matrix_group_corollary}}
\label{sec:general_case}

We will prove 
Theorems~\ref{thm:gcd} and \ref{thm:universal}
by reducing to the connected case,
where we can apply Theorem~\ref{thm:strong_swop_grp}.
Recall the definition of \(\GpRg\) from Definition~\ref{def:S0}.

\begin{lemma}\label{lem:projection}
Let \(S\) and \(T\) be orders with $S$ non-zero, let \(R=S\times T\) with projection map
\(\pi\colon R\to S\), and let \((A,G)\in\GpRg(R)\). Then we have
\((\pi(A),\pi(G))\in\GpRg(S)\) and the restriction \(G\to\pi(G)\) of $\pi$ is a
group isomorphism.
\end{lemma}
\begin{proof}
We have a natural map \(\pi(A)[G]\twoheadrightarrow \pi(A)[\pi(G)]\to S\).
Since \(S\) equals \(\pi(A[G])=\sum_{g\in G} \pi(A) \pi(g)\), this map is clearly surjective.
Suppose \(\sum_{g\in G} \pi(a_g) g\) is in its kernel. Writing $e=(1,0)\in R$ and identifying $S$
with $S\times\{0\}$, we have \(\pi(x)=ex\) for all \(x\in R\). By
Proposition~\ref{prop:reduced_connected_groupring}(ii) we have $e\in A$ and therefore
\(\sum_{g\in G} ea_g g=0\) in \(A[G]\). We conclude that for all \(g\in G\) we have \(\pi(a_g)=ea_g=0\),
so
the map \(\pi(A)[G]\to S\) is an isomorphism.
Then the maps \(\pi(A)[G]\to \pi(A)[\pi(G)]\) and \(\pi(A)[\pi(G)]\to S\) are isomorphisms as well.
Since $S\ne0$, this implies that the map $G\to\pi(G)$ is an isomorphism and that 
\((\pi(A),\pi(G))\in \GpRg(S)\).
\end{proof}

Given a group ring structure on a product of orders, Lemma~\ref{lem:projection} constructs
on each of the factors a group ring structure, with the same group. The following proposition does the
opposite. For the definition of greatest common divisors, see Remark~\ref{rem:gcd}.

\begin{proposition}\label{prop:join}
Let \(X\) be a finite non-empty set, let \((R_x)_{x\in X}\) be a collection of connected
orders,  and let \((A_x,G_x)\in \GpRg(R_x)\) for all \(x\in X\).
Suppose that for all \(x\in X\) we write \(G_x= D_x\oplus E_x\) for some subgroups
\(D_x\), \(E_x\subset G_x\) such that for all \(x\), \(y\in X\) we have \(D_x\cong D_y\). 
Put \(R=\prod_{x\in X}R_x\) and \(A=\prod_{x\in X} A_x[E_x]\), and let
\(D\subset \prod_{x\in X} D_x\) be a subgroup for which all 
the 
projection maps
\(\pi_x\colon D\to D_x\) are isomorphisms. Then 
\((A,D)\in \GpRg(R)\).
If in addition \((A_x,G_x)\) is maximal in \(\GpRg(R_x)\) for all \(x\in X\), and
$D$ is a greatest common divisor of $\{G_x:x\in X\}$, then \((A,D)\) is maximal in
\(\GpRg(R)\).
\end{proposition}
\begin{proof}
Clearly 
\(A\subset R\) and \(D\subset \mu(R)\). There is a sequence of ring isomorphisms
\[A[D]\cong \prod_{x\in X}(A_x[E_x][D])\cong\prod_{x\in X} (A_x[E_x][D_x]) \cong
\prod_{x\in X} A_x[G_x] = R,\]
where one obtains the first isomorphism by tensoring $A=\prod_{x\in X} A_x[E_x]$ with $\Z[D]$ over
$\Z$ and the second isomorphism is induced by the group isomorphisms~$\pi_x$.
The resulting isomorphism \(A[D]\to R\) restricts to the inclusion on both \(A\) and \(D\), so  
\(A[D]=R\) and indeed \((A,D)\in \GpRg(R)\).

Now suppose that \((A_x,G_x)\) is maximal in \(\GpRg(R_x)\) for all \(x\in X\), and that \(D\)
is a greatest common divisor of \(\{G_x:x\in X\}\). 
Let \((B,H)\in \GpRg(R)\) be such that \((A,D)\leq (B,H)\). 
For \(x\in X\) let \(B_x\) and \(H_x\) be the projection of \(B\), respectively \(H\), to \(R_x\),
so 
by Lemma~\ref{lem:projection} we have \((B_x,H_x)\in\GpRg(R_x)\) and \(H\cong H_x\).
Choose \((C_x,I_x)\in \GpRg(R_x)\) to be maximal such that \((B_x,H_x)\leq (C_x,I_x)\).
Since 
$R_x$ 
is connected, Lemma~\ref{lem:connected_factorization} implies that there exists a finite
abelian group \(F_x\) such that \(I_x \cong H_x \oplus F_x\). 
Since both \((A_x,G_x)\) and \((C_x,I_x)\) are maximal in \(\GpRg(R_x)\), we have
\(G_x\cong I_x\) by Theorem~\ref{thm:strong_swop_grp}.
Hence \(G_x\cong I_x \cong H_x \oplus F_x\cong H\oplus F_x\). 
Thus \(H\) is a common divisor of all \(G_x\), and \(H\) contains~\(D\). Since \(D\) is a greatest
common divisor, we obtain \(H=D\). From $A[D]=B[H]=B[D]$ and $A\supset B$ we see $A=B$, so 
\((A,D)=(B,H)\) 
and \((A,D)\) is maximal.
\end{proof}

\noindent{\bf{Proof of Theorem~\ref{thm:gcd}.}}
If \(A=0\) or \(B=0\), then 
Theorem~\ref{thm:gcd} 
holds trivially. Hence assume \(A\) and \(B\) are non-zero.

(ii) \(\Rightarrow\) (i) Assuming (ii), we have ring isomorphisms
\[A[G]\cong C[I][G] \cong C[I\times G]\cong C[J\times H]\cong C[J][H]\cong B[H].\]
(i) \(\Rightarrow\) (ii) 
First assume \(A[G]\) is connected.
Let \((C,V)\geq (A,G)\) and \((D,W)\geq (B,H)\) be 
a 
maximal element 
of \(\GpRg(A[G])\),
respectively \(\GpRg(B[H])\). By Lemma~\ref{lem:stark_iff_maximal} the orders $C$ and $D$
are stark, so by Theorem~\ref{thm:strong_swop_grp} there exists a ring isomorphism
\(\sigma\colon B[H]\to A[G]\) that sends \((D,W)\) to \((C,V)\).
It follows that \((C,V)\geq (\sigma(B),\sigma(H))\), so applying Lemma~\ref{lem:connected_factorization}
twice, we find subgroups \(I\), \(J\subset V\) such that \(I\times G=V=J\times\sigma(H)\cong J\times H\) 
and 
\(C[I]=A\) and \(C[J]=\sigma(B)\cong B\). This concludes the proof of the connected case.

Next consider the general case, where \(A[G]=\prod_{x\in X}R_x\) is a non-empty product of connected
reduced orders~$R_x$. Without loss of generality we may assume \(A[G]=B[H]\). Let $x\in X$.
Write \(A_x\) and \(B_x\) for the image of \(A\), respectively \(B\), of the projection onto \(R_x\).
Then  
\(A_x[G]\cong R_x\cong B_x[H]\) by Lemma~\ref{lem:projection}.
Since \(R_x\) is connected and we proved (i) \(\Rightarrow\) (ii)  
in the connected case, there
exist a reduced order \(C_x\) and finite abelian groups \(I_x\) and \(J_x\) such that
\(C_x[I_x]\cong A_x\) 
and 
\(C_x[J_x]\cong B_x\) and \(I_x\times G \cong J_x\times H=P_x\).
Replacing \(C_x\) by \(C_x[D_x]\) for some greatest common divisor \(D_x\) of \(I_x\) and \(J_x\),
we may assume that \(I_x\) and \(J_x\) are coprime. It follows that \(P_x\) is a least common
multiple of \(G\) and \(H\), as defined in Remark~\ref{rem:gcd}. In particular, when $x$ ranges over $X$,
the finite abelian groups \(P_x\) are pairwise isomorphic, and as a consequence the same holds for
the groups \(I_x\). Hence there exists a subgroup \(I\subset \prod_{x\in X} I_x\) such that all
projections \(I\to I_x\) are isomorphisms, so  
from Proposition~\ref{prop:join} it follows that
\(C[I]\cong A\) with \(C=\prod_{x\in X} C_x\). Similarly we find a finite abelian group \(J\) that
is isomorphic to all $J_x$ such that \(C[J]\cong B\). Now $I$ and $J$ together satisfy
\(I \times G \cong J \times H\), as desired. \qed \\

\noindent{\bf{Proof of Theorem~\ref{thm:universal}.}}
Let \((A,G)\in\GpRg(R)\) be a maximal element (Lemma~\ref{lem:max_exists}).
Then \(A\) is stark by Lemma~\ref{lem:stark_iff_maximal}.
Suppose \(B\) is a stark ring and \(H\) is a finite abelian group such that \(B[H]\cong R\).
By Theorem~\ref{thm:gcd} there exist an order \(C\) and finite abelian groups \(I\) and \(J\)
such that \(A\cong C[I]\)  
and \(B\cong C[J]\) and \(I\times G\cong J\times H\). 
Since both \(A\) and \(B\) are stark we conclude 
that 
\(I=J=1\), so 
\(G\cong H\) and
\(A\cong C\cong B\). Hence \(A\) and \(G\) are unique up to ring and group isomorphism, respectively. \qed \\

\noindent{\bf{Proof of Theorem~\ref{thm:matrix_group_corollary}.}}
We use the notation of Theorem~\ref{thm:matrix_group_corollary}. 
By Remark \ref{rem:r2d2}, if $\alpha\in\Aut(A)$ and $\gamma\in\Gamma$ then $\alpha(A_\gamma) = A_{\alpha\gamma}$. 
It then follows readily from Theorem~\ref{thm:matrix_group} that the orbit of \(A\) under \(\Aut(A[G])\)
is $\mathcal{A}$ and the orbit of \(G\) under \(\Aut(A[G])\) is 
$\mathcal{G}$. 
Theorem~\ref{thm:matrix_group_corollary} now follows from Theorem~\ref{thm:universal}. \qed

\section{Algorithms}\label{sec:algo}

In this section we will prove Theorem~\ref{thm:algos}, the algorithmic counterpart to Theorem~\ref{thm:universal}.
To state our theorems rigorously, we should specify how our data are encoded. 
We will continue the conventions used in \cite{RootsOfUnity,algorithmsQ,
UniversalGradings,Lattices,Iuliana,Daan}, the results from which are therefore compatible.
They can be briefly described as follows. Integers and rationals are encoded straightforwardly. 
Any finitely generated free module is specified by its rank, its elements are represented by vectors,
and morphisms between such modules are represented by matrices.
All rings we consider are Noetherian, so each finitely generated module is a cokernel of some
morphism of finitely generated free modules, and it is encoded as such.
Finally, a ring structure on a finitely generated abelian group is encoded as a collection of
multiplication maps, one for every generator. For matrices with integer coefficients we can do
multiplication and compute (bases for) kernels and images in polynomial time as described in
\cite{Lattices}.

\begin{theorem}[Theorem 4.1.1 in \cite{Iuliana}]\label{thm:comp_gcd}
There exists a polynomial-time algorithm that, given a finite ring \(R\) and finite \(R\)-modules \(M_1\) and \(M_2\), computes a greatest common divisor \(D\) of \(M_1\) and \(M_2\) as defined in Remark~\ref{rem:gcd}, together with injections \(\iota_i\colon D\to M_i\) and a complement \(N_i\subset M_i\) such that \(N_i\oplus \iota_i D=M_i\). 
\end{theorem}

\begin{proposition}\label{prop:comp_gcd_inf}
For each of \(R=\Z\) and \(R=\ltm\) there exists a polynomial-time algorithm that, given finite
\(R\)-modules \(M_1\) and \(M_2\), computes a greatest common divisor \(D\) of \(M_1\) and \(M_2\),
together with injections \(\iota_i\colon D\to M_i\) and a complement \(N_i\subset M_i\) such that
\(N_i\oplus \iota_i D=M_i\).
\end{proposition}
\begin{proof}
By Theorem 2.6.9 in \cite{Iuliana} we may compute the exponents of \(M_1\) and \(M_2\), and their
least common multiple \(n\), in polynomial time. Note that \(M_1\) and \(M_2\) are \(R/nR\)-modules
and that replacing $R$ by $R/nR$ does not change the problem. Since
\(R/nR\) is a finite ring, we can thus reduce to Theorem~\ref{thm:comp_gcd}.
\end{proof}

Proposition~\ref{prop:matrix_cat} allows us to interpret a morphism of finite abelian
groups as a finite length \ltm-module. Although both types of objects are represented differently,
one easily deduces from the proof of Proposition~\ref{prop:matrix_cat} that we can change
representations in polynomial time.

In the following result, \(\textup{Dec}_\mathcal{I}(d)\) is as defined in Definition~\ref{def:dec},
Remark~\ref{rem:transposing}.iv, and Definition~\ref{def:I}.

\begin{proposition}\label{prop:comp_S1}
There exists a polynomial-time algorithm that, given finite abelian groups \(A\) and \(B\) and a
morphism \(d\colon A\to B\), computes a maximal element of \(\textup{Dec}_\mathcal{I}(d)\).
\end{proposition}
\begin{proof}
By Proposition~\ref{prop:comp_gcd_inf} we may compute in polynomial time a greatest common divisor
\(D\) of \(A\) and \(B\) as \(\Z\)-modules.
Similarly we may compute a greatest common divisor \(E\) of \(d\) and \(\id_{D}\) as \(\ltm\)-modules.
We also obtain submodules \(d_0\) and \(d_1\) of \(d\) such that \(d_1\cong E\) and \(d=d_0\oplus d_1\).
We claim that \((d_0,d_1)\) is a maximal element of \(\textup{Dec}_\mathcal{I}(d)\).
First note that \(d_1\) is a divisor of \(\id_{D}\) and thus must be an isomorphism. 
As \(d=d_0\oplus d_1\) we indeed have that \((d_0,d_1)\in \textup{Dec}_\mathcal{I}(d)\).
Let \((e_0,e_1)\geq (d_0,d_1)\) be maximal in \(\textup{Dec}_\mathcal{I}(d)\).
Since \(e_1\) is an isomorphism, it is isomorphic to \(\id_F\) for some finite abelian group \(F\).
Since \(e_1\) is a direct summand of \(d\), the group \(F\) is a direct summand of both \(A\) and \(B\), so \(F\) is a divisor of their greatest common divisor \(D\).
Thus \(e_1\) is a divisor of \(\id_D\). 
It follows that \(e_1\) is a divisor of \(E\cong d_1\), so \((d_0,d_1)=(e_0,e_1)\) and thus \((d_0,d_1)\) is maximal, as was to be shown.
\end{proof}

In the following result, we represent a grading $(\Delta,(R_\delta)_{\delta\in\Delta})$ of an
order $R$ by a finitely generated abelian group $\Delta$ by specifying only those subgroups
$R_\delta$ of $R$ that are non-zero. This differs from the representation of gradings used
in \cite{Daan}, which restricts to finite $\Delta$ and specifies {\em all\/} subgroups $R_\delta$;
the latter representation can readily be converted to the former in polynomial time, which is all
we need.

\begin{proposition}\label{prop:comp_S0}
There exists a polynomial-time algorithm that, given a reduced order $R$ and a universal grading
\((\Gamma,\mathcal{R})\) of \(R\), computes a maximal element of \(\GpRg(R)\) as
defined in Definition~\ref{def:S0}.
\end{proposition}
\begin{proof}
First suppose \(R\) is connected.
By Theorem~1.2 in \cite{RootsOfUnity} we may compute \(\mu=\mu(R)\) in polynomial time and thus also the group homomorphism \(d\colon \mu\to\Gamma\) as defined in Definition~\ref{def:degree-map}. 
We may compute a maximal element \((d_0,d_1)\in\textup{Dec}_\mathcal{I}(d)\), with \(d_i\colon \mu_i\to \Gamma_i\) as in Remark~\ref{rem:mui_direct_summand}, in polynomial time using Proposition~\ref{prop:comp_S1}. 
Under the isomorphisms of partially ordered sets of Theorem~\ref{thm:composite_maps} this \(d\) corresponds to a maximal element \((A,G)\in\GpRg(R)\), where \(A=\sum_{\gamma\in \Gamma_0} R_\gamma\) and \(G=\mu_1\), which we may compute in polynomial time.

Now consider the general case.
By Theorem~1.1 in \cite{RootsOfUnity} we may compute in polynomial time connected reduced orders \((R_x)_{x\in X}\) for some index set \(X\) such that \(R\cong \prod_{x\in X} R_x\), together with the projections \(\pi_x\colon R\to R_x\).
Using Proposition~\ref{prop:product_grading}.i we may construct universal gradings for the \(R_x\) in polynomial time. 
Hence by the special case we may compute a maximal element of \(\GpRg(R_x)\) for all \(x\in X\) in polynomial time.
Finally, we may apply Proposition~\ref{prop:join} to compute a maximal element of \(\GpRg(R)\), observing that the construction in Proposition~\ref{prop:join} can be carried out in polynomial time using Proposition~\ref{prop:comp_gcd_inf}.
\end{proof}

Computing a maximal element of \(\GpRg(R)\) for a reduced order \(R\) is now reduced to finding a universal grading of \(R\). 

\begin{theorem}[Theorem~1.4 in \cite{Daan}]\label{thm:daan_extern}
There is an algorithm that takes a reduced order \(R\) as input and produces a universal grading of \(R\) in time \(n^{O(m)}\), where \(n\) is the length of the input and \(m\) is the number of minimal prime ideals of~\(R\).
\end{theorem}

The above algorithm does not, in general, run in polynomial time;
but when we bound \(m\), or equivalently the number $2^m$ of idempotents of \(R\tensor_\Z \Q\), by a constant
it is guaranteed to finish in polynomial time.

We now consider the case where \(R\) is generated as a group by its autopotent elements.
We recall that a ring element \(x\) is called autopotent if \(x^{n+1}=x\) for some $n\in\Z_{>0}$. 
Write \(\alpha(R)\) for the set of autopotents of \(R\).

\begin{proposition}\label{prop:autopotents}
Let $S$ and $R$ be rings. Then: 
\begin{enumerate}[nosep]
\item The roots of unity and idempotents of \(R\) are autopotent.
\item The product of any two commuting autopotents of $R$ is autopotent.
\item We have \(\mu(R\times S)=\mu(R)\times\mu(S)\) and \(\alpha(R\times S)=\alpha(R)\times\alpha(S)\).
\item Let $x\in R$. Then $x\in\alpha(R)$ if and only if there exist an idempotent $e\in R$ and a $\zeta\in \mu(R)$ such that $x=e\zeta=\zeta e$.
\item If $R$ is commutative, then $R$ is 
generated 
as a ring 
by \(\alpha(R)\) if and only if its
additive group is generated by \(\alpha(R)\).
\item As groups, $R\times S$ is generated by autopotents if and only if each of $R$ and $S$ is generated by autopotents.
\item If $R$ is connected, then \(\alpha(R)=\mu(R)\cup\{0\}\).
\end{enumerate}
\end{proposition}
\begin{proof}
Statements (i), (ii) and (iii) are trivial.  
The `if'-part
of (iv) 
follows from (i) and (ii).
Conversely, suppose \(x^{n+1}=x\). Then \(e=x^{n}\) satisfies \(e^2=e\), so \(e\) is idempotent. 
Assume without loss of generality that \(R=\Z[x]\), so 
\(R\) is commutative.
Hence we may decompose \(R=eR\times(1-e)R\). As \(x\in eR\) is an \(n\)-th root of unity, so is \(\zeta=(x,1)\in R\). 
Then \(x=e\zeta=\zeta e\).

By (ii) 
the set of autopotents is 
closed under multiplication,
and this gives (v).
Part (vi) 
follows from (iii) combined with the fact that \(0\in\alpha(R)\) and \(0\in\alpha(S)\).
Part (vii) follows from (iv).
\end{proof}

The `if'-part of (vi) is wrong when we replace `autopotents' by `roots of unity':
As a group the connected ring \(\Z\) is generated by its roots of unity \(\{\pm1\}\), but for \(\Z\times\Z\) the group generated by the roots of unity does not contain \((1,0)\).
This is the main reason we introduce autopotents.

\begin{lemma}\label{lem:alpha_gen_implies_reduced}
Let \(R\) be an order that is generated as a group by \(\alpha(R)\). Then \(R\) is reduced.
\end{lemma}
\begin{proof}
It suffices to prove that \(K=R\tensor_\Z\Q\) is reduced, because \(R\to K\) is injective.
Each \(x\in\alpha(R)\) has a minimal polynomial in \(K[X]\) dividing \(X^{n+1}-X\) for some \(n > 0\).
In particular \(x\) is separable, and consequently so are all elements of \(K\).
As \(0\) is the only separable nilpotent element, the lemma follows.
\end{proof}

\begin{definition}[Example~3.4 in \cite{UniversalGradings}]\label{def:inner_product}
For an order \(R\) we define a bilinear map
\[ \langle x,y\rangle_R = \langle x,y\rangle = \sum_{\sigma\colon R\to\C} \sigma(x)\cdot\overline{\sigma(y)}, \]
where the sum ranges over all ring homomorphisms, of which there are only finitely many.
\end{definition}

\begin{remark}\label{rem:inner_product_domain}
Following Example~3.4 in \cite{UniversalGradings}, the map from Definition~\ref{def:inner_product} is non-degenerate when \(R\) is reduced, i.e., \(\langle x,x\rangle =0\) implies \(x=0\) for all \(x\in R\).
We have a bijective correspondence
\begin{align*} 
\{ \sigma\colon  R\to \C \} \leftrightarrow \{ (\p,\sigma_\p):\p\subset R\text{ a minimal prime ideal},\ \sigma_\p\colon R/\p\to \C \} 
\end{align*}
that sends \(\sigma\colon R\to\C\) to \((\ker(\sigma),\tilde{\sigma})\) where \(\tilde{\sigma}\colon R/\ker(\sigma)\to\C\) is given by the homomorphism theorem, and conversely sends \((\p,\sigma_\p)\) to \(\sigma_\p\) composed with the projection \(\pi_\p\colon R\to R/\p\).
Thus 
for all \(x,y\in R\)
we have 
\(\langle x,y\rangle_R = \sum_{\p\subset R} \langle \pi_\p(x),\pi_\p(y)\rangle_{R/\p}\), where the sum ranges over all minimal prime ideals.
\end{remark}

\begin{lemma}\label{lem:compute_inner}
For all orders \(R\) that are generated as a group by \(\alpha(R)\) we have \(\langle R,R\rangle\subset \Z\).
There exists a polynomial-time algorithm that, given an order \(R\) that is generated as a group by \(\alpha(R)\) and \(x,y\in R\), computes \(\langle x,y\rangle\).
\end{lemma}
\begin{proof}
Note that \(R\) is reduced by Lemma~\ref{lem:alpha_gen_implies_reduced}.
Let \(X\) be the set of minimal primes of \(R\).
Using Theorem~1.10 in \cite{algorithmsQ} we may compute \(X\) and for each \(\p\in X\) the map \(R\to R/\p\) in polynomial time.
Note that as a group, \(R/\p\) is
generated by \(\alpha(R/\p)\).
Then by the formula of Remark~\ref{rem:inner_product_domain} it suffices to prove the lemma for the ring \(R/\p\). 
Thus we suppose \(R\) is a domain and consequently \(\alpha(R)=\mu(R)\cup\{0\}\) by Proposition~\ref{prop:autopotents}.vii. 
For \(\zeta,\xi\in\mu(R)\) and 
a ring homomorphism 
\(\sigma\colon R\to\C\)
we have \(\sigma(\zeta)\cdot\overline{\sigma(\xi)}=\sigma(\zeta\xi^{-1})\).
Thus \(\langle \zeta,\xi\rangle_R = \sum_{\sigma\colon R\to\C} \sigma(\zeta\xi^{-1})\), which is the trace of \(\zeta\xi^{-1}\) from \(R\) to \(\Z\),
and hence 
is an integer.
As \(R\) is generated as a group by \(\mu(R)\), it follows that \(\langle R,R\rangle\subset\Z\) as well.
Moreover, this shows that computing \(\langle x,y\rangle_R\) reduces to computing traces of roots of unity, which clearly can be done in polynomial time.
\end{proof}

For a ring \(R\), an \(R\)-module \(M\), and a subset \(X\subset M\), we write \(R\cdot X\) for the submodule of \(M\) generated by~\(X\).

\begin{lemma}\label{lem:compute_gen_Q}
There exists a polynomial-time algorithm that, given a finite-dimensional commutative \(\Q\)-algebra \(A\) and a finite set \(X\subset A\), computes a \(\Q\)-basis \(Y\) of the subalgebra \(B\) of \(A\) generated by \(X\), where each element in \(Y\) is a finite (possibly vacuous) product of elements of \(X\).
\end{lemma}
\begin{proof}
The algorithm proceeds as follows.
Start with \(Y=\{1\}\).
Compute the set of products \(Z=\{xy:x\in X,\, y\in Y\}\) and update \(Y\) to be a maximal \(\Q\)-linearly independent subset of \(Z\cup Y\).
Repeat this until \(\Q\cdot Y\) is stable.

Suppose in some step \(\Q\cdot Y = \Q \cdot (Z\cup Y)\). Then \(Z\subset \Q\cdot Y\), so \(\Q\cdot Y\) is closed under taking products with~\(X\). 
Since \(X\) generates \(B\) as 
a 
\(\Q\)-algebra and \(1\in\Q \cdot Y\) by the choice of initial \(Y\), it follows that \(\Q\cdot Y=B\).
Note that \(\# Y\leq \dim_\Q(B)\) and thus there are at most \(\dim_\Q(B)\) steps in the algorithm.
Moreover, in each step \(\# Z \leq \# (X\times Y)\) is polynomially bounded in the input length, so in total there are only polynomially many multiplications.
Lastly, note that in step \(i\) of the algorithm each element of \(Y\) can be written as a product of \(i\) elements from \(X\), and therefore the encoding of every element has length proportional to at most \(i\) times that of the longest element of \(X\).
Hence the multiplications can be carried out in polynomial time.
\end{proof}

Although it is possible to compute \(\alpha(R)\) for a reduced order
$R$, 
we cannot in general do this in polynomial time, even if \(R\) is connected. 
Note that 
 for the ring 
 $$R=\{(a_i)_i \in \Z^{n} : (\forall i,j)\ a_i \equiv a_j \textup{ mod }2\},$$
the set 
\(\{-1,1\}^{n} = \mu(R) \subset \alpha(R)\) is exponentially large.
Theorem~1.4 of \cite{RootsOfUnity} gives a polynomial-time algorithm that, given an order $R$, produces a set of generators of $\mu(R)$.

\begin{proposition}\label{prop:compute_gen_unity}
There exists a polynomial-time algorithm that, given an order \(R\), computes a set \(Y\subseteq \alpha(R)\) such that \(\Z \cdot Y = \Z \cdot \alpha(R)\).
\end{proposition}
\begin{proof}
We may factor \(R\) into a product of connected orders in polynomial time using Algorithm~6.1 in \cite{RootsOfUnity}. 
Combined with Proposition~\ref{prop:autopotents}.vii we may assume \(R\) is connected and \(\alpha(R)=\mu(R)\cup\{0\}\).

Apply Theorem~1.2 in \cite{RootsOfUnity} to compute in polynomial time a set \(X\) of generators of the group \(\mu(R)\).
Using Lemma~\ref{lem:compute_gen_Q} we may compute a basis \(Z\subseteq\mu(R)\) for the subalgebra \(\Q\cdot\mu(R)\) of \(R\tensor\Q\) as a \(\Q\)-vector space. 

Denote the discriminant $\det( (\Tr_{\Q\cdot \mu(R)/\Q} (xy))_{x,y\in Z})$ of \(\Z\cdot Z\) by $\Delta_Z$, and similarly let $\Delta_{\mu(R)}$ denote the discriminant of \(\Z\cdot \mu(R)\).
Let \(n=\#Z=\dim_\Q(\Q\cdot \mu(R))\).
We have \(|\Delta_Z|\leq n^{3n/2}\)
by Hadamard's inequality and the fact that \(|\Tr(\zeta)|\leq n\) for \(\zeta\in \mu(R)\).
Thus,
$$
\#(\Z\cdot\mu(R)/\Z\cdot Z)^2 = |\Delta_Z|/|\Delta_{\mu(R)}| \le |\Delta_Z| \le n^{3n/2}.
$$
In particular, \(\log_2\#(\Z\cdot\mu(R)/\Z\cdot Z)\) is polynomially bounded.

First we set \(Y=Z\).
Then we iterate over \(x\in X\) and \(y\in Y\) and add \(xy\) to \(Y\) whenever \(xy\not\in\Z\cdot Y\).
Once \(\Z \cdot Y\) stabilizes we have \(\Z\cdot Y = \Z\cdot \mu(R)\) and may return \(Y\).
Each new element added to \(Y\) decreases \(\log_2 \#(\Z\cdot \mu(R)/\Z \cdot Y)\) by at least \(1\), so the cardinality of \(Y\) and the number of steps taken in the algorithm are polynomially bounded.
Finally, we remark that there is a polynomial upper bound on the lengths of the encodings of the elements of \(Y\), since each element is the product of at most \(\#Y\) elements of \(X\) and an element of \(Z\).
Hence the algorithm runs in polynomial time.
\end{proof}

Roots of unity are homogeneous in any grading of a connected order, by Theorem 1.5 of \cite{UniversalGradings}. The next result shows this is also true in any orthogonal decomposition of a connected reduced order, where orthogonal means with respect to the inner product in Definition~\ref{def:inner_product}.

\begin{lemma}\label{lem:roots_of_unity_indec}
Let \(\mathcal{R}\) be an orthogonal decomposition of a connected reduced order \(R\).
Then the roots of unity of \(R\) are homogeneous in \(\mathcal{R}\), i.e., for all \(\zeta\in\mu(R)\) there exists some \(M\in\mathcal{R}\) such that \(\zeta\in M\).
\end{lemma}
\begin{proof}
By Corollary~5.6 in \cite{UniversalGradings} the element \(1\in R\) is {\em indecomposable}, i.e., for all \(x,y\in R\) such that \(x+y=1\) and \(\langle x,y\rangle\geq 0\) we have \(x=0\) or \(y=0\). Since \(x\mapsto \zeta x\) is an isometry of \(R\) for all \(\zeta\in\mu(R)\), we conclude that all \(\zeta\in\mu(R)\) are indecomposable.
Indecomposable elements are 
clearly 
homogeneous.
\end{proof}

\begin{lemma}\label{lem:universal_of_mu_gen}
Let \(R\) be a connected reduced order and \(\overline{R}=(\Gamma,(R_\gamma)_\gamma)\) be the universal grading
of $R$.
Suppose \(X\subset\mu(R)\) is a subset and \(A\subset R\) is a subring such that \(\bigoplus_{x\in X} Ax\) is an orthogonal decomposition of \(R\). 
Write \(\Delta=\mu(R)/\mu(A)\).
Then the natural map \(g\colon X\to\Delta\) is a bijection and \(\overline{S}=(\Delta,(A \cdot g^{-1}(\delta))_{\delta})\) is a grading of \(R\).
If also \(A\subset R_1\), then \(\overline{S}\) is universal.
\end{lemma}
\begin{proof}
First we show \(g\) is a bijection. 
If \(g(x)=g(y)\) for \(x,y\in X\), then \(Ax=Ay\) and thus \(x=y\) by orthogonality. Hence \(g\) is injective. 
It follows from Lemma~\ref{lem:roots_of_unity_indec} 
that \(\mu(R) = \bigcup_{x\in X} ( Ax \cap \mu(R) ) = \mu(A) \cdot X\), so \(g\) is surjective.
It follows that \(\overline{S}\) is a grading
of $R$.
Now assume \(A\subset R_1\).
Let \(e\colon \Delta\to\Gamma\) be induced by the degree map \(d\colon \mu(R)\to\Gamma\), which is well-defined because \(\mu(A)\subset\ker(d)\) by assumption.
Then 
\(e_* (\overline{S})=\overline{R}\), and \(\overline{S}\) is universal by Proposition~\ref{prop:universal_efficient}.iv.
\end{proof}

\begin{theorem}\label{thm:universal_unity_grading}
There exists a polynomial-time algorithm that, given a reduced order \(R\) that is generated as a group by \(\alpha(R)\), computes the universal grading of \(R\).
\end{theorem}
\begin{proof}
We may write \(R\) as a product of connected orders in polynomial time using Algorithm~6.1 in \cite{RootsOfUnity}. 
Using Proposition~\ref{prop:product_grading}.ii and Proposition~\ref{prop:autopotents}.vi we may restrict ourselves to the connected case, so suppose \(R\) is connected and thus \(\alpha(R)=\mu(R)\cup\{0\}\) by Proposition~\ref{prop:autopotents}.vii.
Let \(\overline{R}=(\Gamma,(R_\gamma)_\gamma)\) be the universal grading of \(R\), which is not part of the algorithm.

We will compute a subring \(A\subset R\) and a sequence \(s_1,\dotsc,s_n\in \mu(R)\) having the following properties:
\begin{align*}
\textup{(i)} \enspace \sum_{i=1}^n A s_i = R; \hspace{1.5em} \textup{(ii)} \enspace A \subset R_1; \hspace{1.5em} \textup{(iii)} \enspace (\forall\, i\neq j)\ \langle A s_i, A s_j \rangle = 0.
\end{align*}
It then follows from Lemma~\ref{lem:universal_of_mu_gen} that we have in fact computed the universal grading of \(R\).
The algorithm is as follows.
\begin{enumerate}[nosep,label={(\arabic*)}]
\item Set \(A=\Z\) and compute \(\Z\)-module generators \(s_1,\dotsc,s_n\in\mu(R)\) for \(R\) using Proposition~\ref{prop:compute_gen_unity}.
\item While there exist \(i\neq j\) such that \(\langle As_i, As_j \rangle \neq 0\), as computed by Lemma~\ref{lem:compute_inner}:
\begin{enumerate}[nosep]
\item Choose any \(i<j\) such that \(\langle A s_i, As_j\rangle\neq 0\);
\item Replace \(A\) by the ring generated by \(A\) and \(s_i^{-1}s_j\);
\item Remove \(s_j\) from the list of generators.
\end{enumerate}
\end{enumerate}
Clearly,
properties 
(i) and (ii) are satisfied after step (1), and (iii) is satisfied after step (2).
Moreover, since the number of generators decreases each step, the number of iterations in step (2) is polynomially bounded and the algorithm terminates.
It remains to show 
that 
(i) and (ii) are preserved by step (2).
Property (i) is preserved because \(s_j\) is contained in \(A s_i\) after updating \(A\) in step (2b).
For (ii) it suffices to show that \(s_i^{-1}s_j\in R_1\).
As \(s_i^{-1}s_j\in\mu(R)\) we have \(s_i^{-1}s_j\in R_\gamma\) for \(\gamma=d(s_i^{-1}s_j)\) by Lemma~\ref{lem:degree_abhom}.
Since \(\langle A,As_i^{-1}s_j\rangle=\langle As_i,As_j\rangle\neq 0\) we conclude that \(R_\gamma=R_1\) by Proposition~5.8 in \cite{UniversalGradings}, as was to be shown.
Lastly, note that all operations can be carried out in polynomial time.
\end{proof}

\noindent{\bf{Proof of Theorem~\ref{thm:algos}.}}
By Proposition~\ref{prop:comp_S0} it suffices to compute a universal grading of \(R\).
Apply Theorem~\ref{thm:daan_extern}
in the general case 
and Theorem~\ref{thm:universal_unity_grading} in the specific case of an order generated by its autopotents.

\bibliography{references}{}
\bibliographystyle{plain}

\end{document}